\documentclass[a4paper,reqno,11pt]{amsart}
\usepackage{amsfonts,amsmath,amsthm,amssymb,stmaryrd,mathrsfs}
\newtheorem{theorem}{Theorem}[section]
\newtheorem{lemma}[theorem]{Lemma}

\newtheorem{Remark}[theorem]{Remark}

\newtheorem{Corollary}[theorem]{Corollary}
\numberwithin{equation}{section}
\allowdisplaybreaks

\usepackage[left=1 in, right=1 in,top=1 in, bottom=1 in]{geometry}

\newcommand{\na}{\nabla}

\newcommand{\al}{\alpha}

\newcommand{\la}{\lambda}

\newcommand{\pa}{\partial}

\providecommand{\norm}[1]{\left\Vert#1\right\Vert}

\providecommand{\norms}[1]{\left\vert#1\right\vert}
\def\r3{\mathbb{R}^3}

\begin{document}
\title[Compressible non-isentropic Euler-Maxwell system]{The well-posedness of the compressible non-isentropic Euler-Maxwell system in $\r3$}

\author{Zhong Tan}
\address{School of Mathematical Sciences\\
Xiamen University\\
Xiamen, Fujian 361005, China}
\email[Z. Tan]{ztan85@163.com}

\author{Yong Wang}
\address{School of Mathematical Sciences\\
Xiamen University\\
Xiamen, Fujian 361005, China}
\email[Y. Wang]{wangyongxmu@163.com}



\thanks{Corresponding author: Yong Wang, wangyongxmu@163.com}

\thanks{Supported by the National Natural Science Foundation of China--NSAF (No. 10976026) and the National Natural
Science Foundation of China (Grant No. 11271305).}

\begin{abstract}
We first construct the global unique solution by assuming that the initial data is small in the $H^3$ norm but the higher order derivatives could be large. If further the initial data belongs to $\Dot{H}^{-s}$ ($0\le s<3/2$) or $\dot{B}_{2,\infty}^{-s}$ ($0< s\le3/2$), we obtain the various decay rates of the solution and its higher order derivatives. In particular, the decay rates of the density and temperature of electron could reach to $(1+t)^{-\frac{13}{4}}$ in $L^2$ norm.
\end{abstract}

\maketitle

\section{Introduction}
In the present paper, we consider the compressible non-isentropic Euler-Maxwell system (nonconservative form)
\cite{C,J,MRS}
\begin{equation}  \label{yiyi}
\left\{
\begin{array}{lll}
\displaystyle\partial_t\rho+{\rm div} (\rho u)=0,   \\
\displaystyle\partial_tu+u+u\cdot\na u+\na \Theta+\Theta\na\ln\rho=-(E+u\times \tilde{B}), \\
\displaystyle\partial_t\Theta+u\cdot\na \Theta+\frac23\Theta {\rm div}u=\frac13|u|^2-(\Theta-1), \\
\partial_t E-\na \times \tilde{B}=\rho u, & & \\
\partial_t \tilde{B}+\na \times E=0,  \\
{\rm div} E=1-\rho,\ \ {\rm div} \tilde{B}=0, \\
(\rho,u,\Theta,E,\tilde{B})|_{t=0}=(\rho_{0},u_{0},\Theta_0,E_0,\tilde{B}_0),\ x\in\mathbb{R}^3.
\end{array}
\right.
\end{equation}
The unknown functions $\rho, u, \Theta, E, \tilde{B} $ represent the electron density,
electron velocity, absolute temperature, electric field and magnetic field, respectively. In the motion of the fluid, due to the greater inertia the ions merely provide a constant charged background.

Although the compressible Euler-Maxwell system is more and more important in the researches of plasma physics and semiconductor physics, a small amount of results are obtained since its mathematical complexity. In a unipolar form:  Chen, Jerome and Wang \cite{CJW} showed the one-dimensional global existence of entropy weak solutions to the initial-boundary value problem for arbitrarily large initial data in $L^{\infty}(\mathbb{R})$; Guo and Tahvildar-Zadeh \cite{GT}
showed a blow-up criterion for spherically symmetric Euler-Maxwell system;
Recently, there are some results on the global existence
and the large time behavior of smooth solutions with small perturbations, see Tan et al. \cite{TWW}, Duan \cite{D}, Ueda and Kawashima \cite{UK},  Ueda et al. \cite{UWK}; For the asymptotic limits that derive simplified models starting from the Euler-Maxwell system, we refer to \cite{HP,PWG,X} for the relaxation limit, \cite{X} for the non-relativistic limit, \cite{PW1,PW2} for the quasi-neutral limit, \cite{T1,T2} for WKB asymptotics and the references therein.
In a bipolar form: Duan et al. \cite{DLZ} showed the global existence and time-decay rates of solutions near constant steady states with the vanishing electromagnetic field; Xu et al. \cite{XXK} studied the well-posedness in critical Besov spaces. Since the unipolar or bipolar Euler-Maxwell system is a symmetrizable hyperbolic system, the Cauchy problem in $\r3$ has a local unique smooth solution when the initial data is smooth, see Kato \cite{K} and Jerome \cite{J} for instance. Besides, we can refer to \cite{FWK,WFL} for the non-isentropic case.

In this paper, we will refine a global existence of smooth solutions near the constant equilibrium $(1,0,1,0,B_\infty)$ to the compressible non-isentropic Euler-Maxwell system and show some various time decay rates of the solution as well as its spatial derivatives of any order. We should highlight that our results highly depend on the relaxation terms of velocity and temperature. The non-relaxation case is much more difficult, we refer to \cite{GM,G} for such a case. Compared with the compressible isentropic Euler-Maxwell system \cite{TWW}, there are two main difficulties except the computational complexity. The first difficulty is that we must obtain the symmetric hyperbolic non-isentropic Euler-Maxwell system to do the effective energy estimates. We solve this problem by making good use of the positive upper and lower bounds of density and temperature \eqref{n}. The other difficulty is how to deduce the higher decay rates of density and temperature, where is different from the isentropic case \cite{TWW} since the influence of temperature. To overcome this obstacle, we extract two new systems \eqref{npsi yi} and \eqref{npsitheta yi} from the system \eqref{yi}, then we use a bootstrap method to derive \eqref{further decay2}.

Set$$n=\rho-1,\ \ \theta=\Theta-1,\ \ B=\tilde{B}-B_\infty.$$
Then the Euler-Maxwell system \eqref{yiyi} is reformulated equivalently as
\begin{equation}  \label{yi}
\left\{
\begin{array}{lll}
\displaystyle\partial_tn=-u\cdot\na n-(1+n){\rm div} u,   \\
\displaystyle\partial_tu+u+E+\na \theta+u\times B_{\infty}=-u\cdot\na u-\frac{1+\theta}{1+n}\na n-u\times B, \\
\displaystyle\partial_t\theta+\theta=-u\cdot\na \theta-\frac23(1+\theta) {\rm div}u+\frac13|u|^2, \\
\partial_t E-\na \times B-u=n u, & & \\
\partial_t B+\na \times E=0,  \\
{\rm div} E=-n,\ \ {\rm div} B=0, \\
(n,u,\theta,E,B)|_{t=0}=(n_{0},u_{0},\theta_0,E_0,B_0),\ x\in\mathbb{R}^3.
\end{array}
\right.
\end{equation}

For $N\ge 3$, we define the energy functional by
\begin{equation*}
\mathcal{E}_N(t):=\sum_{l=0}^{N}\norm{ \na^{l}(n, u,\theta, E, B )}_{L^2}^2
\end{equation*}
and the corresponding dissipation rate by
\begin{equation*}
\mathcal{D}_N(t):=\sum_{l=0}^{N}\norm{\na^{l}(n,u,\theta)}_{L^2}^2+\sum_{l=0}^{N-1}\norm{\na^{l}E}_{L^2}^2+ \sum_{l=1}^{N-1}\norm{\na^{l} B}_{L^2}^2.
\end{equation*}

Our first main result about the global unique solution to the system \eqref{yi} is stated as follows.
\begin{theorem}\label{existence}
Assume the initial data satisfy the compatible conditions
\begin{equation*}
{\rm div}  {E}_0=-n_0,\ \ {\rm div}  {B}_0=0.
\end{equation*}
There exists a sufficiently small $\delta_0>0$ such that if $\mathcal{E}_3(0)\le \delta_0$, then there exists a unique global solution
$(n,u,\theta,E,B)(t)$ to the Euler-Maxwell system \eqref{yi} satisfying
\begin{equation}\label{energy inequality}
\sup_{0\leq t\leq \infty }\mathcal{ E}_3(t)+\int_{0}^{\infty }%
\mathcal{D}_3(\tau)\,d\tau\leq C\mathcal{ E}_3(0).
\end{equation}

Furthermore, if $\mathcal{E}_N(0)<+\infty$ for any $N\ge 3$, there exists an increasing continuous function $P_N(\cdot)$ with $P_N(0)=0$ such that the unique solution satisfies
\begin{equation}\label{energy inequality N}
\sup_{0\leq t\leq \infty }\mathcal{ E}_N(t)+\int_{0}^{\infty }%
\mathcal{D}_N(\tau)\,d\tau\leq P_N\left(\mathcal{ E}_N(0)\right).
\end{equation}
\end{theorem}

In the proof of Theorem \ref{existence}, the major difficulties are the influence of the temperature and the regularity-loss of the electromagnetic field. We will do the refined energy estimates stated in Lemma \ref{u}--\ref{other di}, which allow us to deduce
\begin{equation*}
\frac{d}{dt} \mathcal{E} _3+\mathcal{D}_3\lesssim\sqrt{\mathcal{E} _3}\mathcal{D}_3
\end{equation*}
and for $N\ge 4$,
\begin{equation*}
\frac{d}{dt} {\mathcal{E}}_N+ \mathcal{D}_N
 \le  C_{N } {\mathcal{D}_{N-1}} {\mathcal{E}_N}.
\end{equation*}
Then Theorem \ref{existence} follows in the fashion of \cite{G12,W12,TWW}.

Our second main result is on some various decay rates of the solution to the system \eqref{yi} by making the much stronger assumption on the initial data.
\begin{theorem}\label{decay}
Assume that $(n,u,\theta,E,B)(t)$ is the solution to the
Euler-Maxwell system \eqref{yi} constructed in Theorem \ref{existence} with $
N\geq 5$. There exists a sufficiently small $\delta_0=\delta_0(N)$ such that if $\mathcal{ {E}}_N(0)\le \delta_0$,
and assuming that $(u_0,\theta_0,E_0,B_0)\in \dot{H}^{-s}$ for some $s\in [0,3/2)$ or $(u_{0},\theta_0,E_0,B_0)\in \dot{B}_{2,\infty}^{-s}$ for some $s\in (0,3/2]$,  then we have
\begin{equation}\label{H-sbound}
\norm{(u,\theta,E,B)(t)}_{\dot{H}^{-s}}\le C_0
\end{equation}or
\begin{equation}\label{H-sbound Besov}
\norm{(u,\theta,E,B)(t)}_{\dot{B}_{2,\infty}^{-s}}\le C_0.
\end{equation}
Moreover, for any fixed integer $k\ge 0$, if $N\ge 2k+2+s$, then
\begin{equation}\label{basic decay}
\norm{\na^k(n,u,\theta,E,B)(t)}_{L^2}\le  C_0 (1+ t)^{- \frac{k+s}{2} }.
\end{equation}

Furthermore, for any fixed integer $k\ge 0$, if $N\ge2k+4+s$, then
\begin{equation}\label{further decay1}
\norm{\na^k(n,u,\theta,E)(t)}_{L^2}\le  C_0 (1+ t)^{- \frac{k+1+s}{2} };
\end{equation}
if $N\ge2k+6+s$, then
\begin{equation}\label{further decay11}
\norm{\na^k n (t)}_{L^2}\le  C_0 (1+ t)^{- \frac{k+2+s}{2} };
\end{equation}
if $N\ge2k+12+s$ and $ B_\infty =0$, then
\begin{equation}\label{further decay2}
\norm{\na^k (n,\theta,{\rm div}u) (t)}_{L^2}\le  C_0 (1+ t)^{- (\frac k2+\frac74+s) }.
\end{equation}
\end{theorem}

In the proof of Theorem \ref{decay}, we mainly use the regularity interpolation method developed in  Strain and Guo \cite{SG06}, Guo and Wang \cite{GW} and Sohinger and Strain \cite{SS}. To prove the optimal decay rate of the
dissipative equations in the whole space, Guo and Wang \cite{GW} developed a general energy
method of using a family of scaled energy estimates with minimum derivative counts and interpolations
among them. However, this method can not be applied directly to the compressible non-isentropic Euler-Maxwell system which is of regularity-loss. To overcome this obstacle caused by the regularity-loss of the electromagnetic field, we deduce from Lemma \ref{u}--\ref{other di} that
\begin{equation*}
\frac{d}{dt} {\mathcal{E}}_k^{k+2}+\mathcal{D}_k^{k+2}\le C_k   \norm{(n,u)}_{L^\infty}\norm{\na^{k+2}(n,u)}_{L^2}
\norm{ \na^{k+2}( E,  B )}_{L^2},
\end{equation*}
where ${\mathcal{E}}_k^{k+2}$ and $\mathcal{D}_k^{k+2}$ with minimum derivative counts are defined by \eqref{1111} and \eqref{2222} respectively.
Then combining the methods of \cite{GW,SS} and a trick of Strain and Guo \cite{SG06} to treat the electromagnetic field, we are able to conclude the decay rate \eqref{basic decay}. If in view of the whole solution, the decay rate \eqref{basic decay} can be regarded as be optimal. The higher decay rates \eqref{further decay1}--\eqref{further decay2} follow by revisiting the equations carefully. In particular, we will use a bootstrap argument to derive \eqref{further decay2}.

By Theorem \ref{decay} and Lemma \ref{Riesz lemma}--\ref{Lp embedding}, we have the following corollary of the usual $L^p$--$L^2$ type of the decay results:
\begin{Corollary}\label{2mainth}
Under the assumptions of Theorem \ref{decay} except that we replace the $\dot{H}^{-s}$ or $\Dot{B}_{2,\infty}^{-s}$ assumption  by that $(u_0,\theta_0,E_0,B_0)\in L^p$ for some $p\in [1,2]$, then for any fixed integer $k\ge 0$, if $N\ge 2k+2+ s_p $, then
\begin{equation*}
\norm{\na^k(n,u,\theta,E,B)(t)}_{L^2}\le  C_0 (1+ t)^{- \frac{k+s_p}{2} }.
\end{equation*}
Here the number $s_{p} :=3\left(\frac{1}{p}-\frac{1}{2}\right)$.

Furthermore, for any fixed integer $k\ge 0$, if $N\ge2k+4+s_p$, then
\begin{equation} \label{p22}
\norm{\na^k(n,u,\theta,E)(t)}_{L^2}\le  C_0 (1+ t)^{- \frac{k+1+s_p}{2} };
\end{equation}
if $N\ge2k+6+s_p$, then
\begin{equation*}
\norm{\na^k n (t)}_{L^2}\le  C_0 (1+ t)^{- \frac{k+2+s_p}{2} };
\end{equation*}
if $N\ge2k+12+s_p$ and $ B_\infty =0$, then
\begin{equation}\label{n L^1}
\norm{\na^k (n,\theta,{\rm div}u) (t)}_{L^2}\le  C_0 (1+ t)^{- (\frac k2+\frac74+s_p) }.
\end{equation}
\end{Corollary}

The followings are several remarks for Theorem \ref{existence}--\ref{decay} and Corollary \ref{2mainth}.

\begin{Remark}
In Theorem \ref{existence}, we only assume that the initial data is small in the $H^3$ norm but the higher order derivatives could be large.
Notice that in Theorem \ref{decay} the $\dot{H}^{-s}$ and $\dot{B}_{2,\infty}^{-s}$ norms of the solution are preserved along the time evolution, however, in Corollary \ref{2mainth} it is difficult to show that the $L^p$ norm of the solution can be preserved.
Note that the $L^2$ decay rate of the higher order spatial derivatives of the solution
is obtained. Then the general optimal $L^q$ $(2\le q\le \infty)$ decay rates of the solution follow by the Sobolev
interpolation.
\end{Remark}

\begin{Remark}
In Theorem \ref{decay}, the space $\Dot{H}^{-s}$ or $\Dot{B}_{2,\infty}^{-s}$ was introduced there to enhance the decay rates.
By the usual embedding theorem, we know that for $p\in (1,2]$, $L^p\subset \Dot{H}^{-s}$ with $s=3(\frac{1}{p}-\frac{1}{2})\in[0,3/2)$. Meantime, we note that the endpoint embedding  $L^1\subset \Dot{B}_{2,\infty}^{-\frac{3}{2}}$ holds. Hence the $L^p$--$L^2(1\le p\le2)$ type of the optimal decay results follows as a corollary.
\end{Remark}

\begin{Remark}
We remark that Corollary \ref{2mainth} not only provides an alternative approach to derive the $L^p$--$L^2$ type
of the optimal decay results but also improves the previous results of the  $L^p$--$L^2$ approach in Feng et al. \cite{FWK}.
In Feng et al. \cite{FWK}, assuming that $B_\infty=0$ and $\norm{(u_0,E_0,B_0)}_{L^1}$ is sufficiently small, by combining the energy method and the linear decay analysis, Feng showed that
\begin{equation*}
\norm{(n,\theta) (t)}_{L^2}\le  C_0 (1+ t)^{- \frac{11}{4} },\ \norm{(u,B)(t)}_{L^2}\le  C_0 (1+ t)^{- \frac{3}{4} } \text{ and } \norm{E(t)}_{L^2}\le  C_0 (1+ t)^{- \frac{5}{4} }.
\end{equation*}
Notice that for $p=1$, our decay rate of $(n,\theta)(t)$ is $(1+t)^{-13/4}$ in \eqref{n L^1}, and $u(t)$ is $(1+t)^{-5/4}$ in \eqref{p22}.
\end{Remark}

\noindent \textbf{Notations:} In this paper, we use $H^{s}(\mathbb{R}^{3}),s\in \mathbb{R}$ to denote the usual
Sobolev spaces with norm $\norm{\cdot}_{H^{s}}$ and $L^{p}(\mathbb{R}^{3}),1\leq p\leq
\infty $ to denote the usual $L^{p}$ spaces with norm $
\norm{\cdot}_{L^{p}}$.
$\na ^{\ell }$ with an
integer $\ell \geq 0$ stands for the usual any spatial derivatives of order $
\ell $. When $\ell <0$ or $\ell $ is not a positive integer, $\na ^{\ell }
$ stands for $\Lambda ^{\ell }$ defined by $
\Lambda^\ell f := \mathscr{F}^{-1} (|\xi|^\ell  \mathscr{F}{f}) $, where $\mathscr{F}$ is the usual Fourier transform operator and $\mathscr{F}^{-1}$ is its inverse. We use $\dot{H}
^{s}(\mathbb{R}^{3}),s\in \mathbb{R}$ to denote the homogeneous Sobolev
spaces on $\mathbb{R}^{3}$ with norm $\norm{\cdot}_{\dot{H}^{s}}$ defined by
$\norm{f}_{\dot{H}^s}:=\norm{\Lambda^s f}_{L^2}$. We then recall the homogeneous Besov spaces. Let
$\phi \in C^\infty_c(R^3_\xi)$ be such that $\phi(\xi) = 1$ when $|\xi| \le 1$ and $\phi(\xi) = 0$ when $|\xi| \ge 2$.  Let
$\varphi(\xi) = \phi(\xi) - \phi(2\xi)$ and $\varphi_j(\xi) = \varphi( 2^{-j}\xi)$
for ${j \in \mathbb {Z}}$. Then by the
construction,
$\sum_{j\in\mathbb {Z}}\varphi_j(\xi)=1$ if $\xi\neq 0.$
We define $\dot{\Delta}_j
f:= \mathscr{F}^{-1}(\varphi_j)* f$, then for $s\in \mathbb{R}$ and $1\le p,r\le \infty$, we define the homogeneous Besov spaces $\dot{B}_{p,r}^{s}(\r3)$ with norm $\norm{\cdot}_{\dot{B}_{p,r}^{s}}$ defined by
$$\|f\|_{\dot{B}_{p,r}^{s}}:=\Big(\sum\limits_{j\in\mathbb{Z}}2^{rsj}\|\dot{\Delta}_j f\|_{L^{p}}^{r}\Big)^{\frac1r}.$$Particularly, if $r=\infty$, then
\begin{equation*}
\|f\|_{\dot{B}_{p,\infty}^{s}}:=\sup\limits_{j\in\mathbb{Z}}2^{sj}\norm{\dot{\Delta}_j f}_{L^{p}}.
\end{equation*}

Throughout this paper, we let $C$  denote
some positive (generally large) universal constants and $\lambda$ denote  some positive (generally small) universal constants. They {\it do not} depend on either $k$ or $N$; otherwise, we will denote them by $C_k$, $C_{N}$, etc.
We will use $a \lesssim b$ if $a \le C b$, and $a\thicksim b$ means that $a\lesssim b$ and $b\lesssim a$.
We use $C_0$ to denote the constants depending on the initial data and $k,N,s$.
For simplicity, we write $\norm{(A,B)}_{X}:=\norm{A}_{X} +\norm{B}_{X}$ and $\int f:=\int_{\mathbb{R}^3}f\,dx.$ $(\ast)\times\varepsilon+(\ast\ast)$ denote that multiplying $(\ast)$ by a sufficiently small but fixed factor $\varepsilon$ and then adding it to $\ast\ast$.

The rest of our paper is structured as follows. In section \ref{section2}, we establish the refined energy estimates for the solution and derive the negative Sobolev and Besov estimates. Theorem \ref{existence} and Theorem \ref{decay} are proved in section \ref{section3}.

\section{Nonlinear energy estimates}\label{section2}
In this section, we will do the a priori estimate by assuming that $\norm{(n,\theta)(t)}_{H^3}\le \delta\ll 1$. Then by Sobolev's inequality, we have
\begin{equation}\label{n}
\frac12\le1+n,\ 1+\theta\le\frac32.
\end{equation}

\subsection{Preliminary}

In this subsection, we collect the analytic tools used later in the paper.

\begin{lemma}\label{A1}
Let $2\le p\le +\infty$ and $\alpha,m,\ell\ge 0$. Then we have
\begin{equation*}
\norm{\na^\alpha f}_{L^p}\le C_{p} \norm{ \na^mf}_{L^2}^{1-\theta}
\norm{ \na^\ell f}_{L^2}^{\theta}.
\end{equation*}
Here $0\le \theta\le 1$ (if $p=+\infty$, then we require that $0<\theta<1$) and $\alpha$ satisfies
\begin{equation*}
\alpha+3\left(\frac12-\frac{1}{p}\right)=m(1-\theta)+\ell\theta.
\end{equation*}
\end{lemma}
\begin{proof}
For the case $2\le p<+\infty$, we refer to Lemma A.1 in \cite{GW}; for the case $p=+\infty$, we refer to Exercise 6.1.2 in \cite{Gla} (pp. 421).
\end{proof}

We recall the following commutator estimate:
\begin{lemma}\label{commutator}
Let $k\ge 1$ be an integer and define the commutator
\begin{equation}\label{commuta}
\left[\na^k,g\right]h=\na^k(gh)-g\na^kh.
\end{equation}
Then we have
\begin{equation*}
\norm{\left[\na^k,g\right]h}_{L^2} \le C_k\left( \norm{\na g}_{L^\infty}
\norm{\na^{k-1}h}_{L^2}+\norm{\na^k g}_{L^2}\norm{ h}_{L^\infty}\right),
\end{equation*}and
\begin{equation*}
\norm{\na^k(gh)}_{L^2} \le C_k\left( \norm{g}_{L^\infty}
\norm{\na^{k}h}_{L^2}+\norm{\na^k g}_{L^2}\norm{ h}_{L^\infty}\right).
\end{equation*}
\end{lemma}
\begin{proof}
It can be proved by using Lemma \ref{A1}, see Lemma 3.4 in \cite{MB} (pp. 98) for instance.
\end{proof}
Notice that when using the commutator estimate in this paper, we usually will not consider the case that $k=0$ since it is trivial.
\begin{lemma}\label{A2}
If the function $f(n)$ satisfies
\begin{equation*}
f(n)\sim n\hbox{ and } \norms{f^{(k)}(n)} \le C_k\hbox{ for any }k\ge 1,
\end{equation*}
then for any integer $k\ge0$, we have
\begin{equation*}
\norm{\na^kf(n)}_{L^2}\le C_k\norm{\na^kn}_{L^2}.
\end{equation*}
\end{lemma}
\begin{proof}
See Lemma 2.2 in \cite{TWW}.
\end{proof}

We have the $L^p$ embeddings:
\begin{lemma}\label{Riesz lemma}
Let $0\le s<3/2,\ 1<p\le 2$ with $1/2+s/3=1/p$, then
\begin{equation*}
\norm{ f}_{\dot{H}^{-s}}\lesssim\norm{ f}_{L^p}.
\end{equation*}
\end{lemma}
\begin{proof}
It follows from the Hardy-Littlewood-Sobolev theorem, see \cite{Gla}.
\end{proof}

\begin{lemma}\label{Lp embedding}
Let $0< s\le 3/2,\ 1\le p<2$ with $1/2+s/3=1/p$, then
\begin{equation*}
\norm{f}_{\dot{B}_{2,\infty}^{-s}}\lesssim\norm{f}_{L^p}.
\end{equation*}
\end{lemma}
\begin{proof}
See Lemma 4.6 in \cite{SS}.
\end{proof}

It is important to use the following special interpolation estimates:
\begin{lemma}\label{1-sinte}
Let $s\ge 0$ and $\ell\ge 0$, then we have
\begin{equation*}
\norm{\na^\ell f}_{L^2}\le \norm{\na^{\ell+1} f}_{L^2}^{1-\theta}%
\norm{ f}_{\dot{H}^{-s}}^\theta, \hbox{ where }\theta=\frac{1}{\ell+1+s}.
\end{equation*}
\end{lemma}
\begin{proof}
It follows directly by the Parseval theorem  and H\"older's
inequality.
\end{proof}

\begin{lemma}\label{Besov interpolation}
Let $s> 0$ and $\ell\ge 0$, then we have
\begin{equation*}
\norm{\na^\ell f}_{L^2}\le \norm{\na^{\ell+1} f}_{L^2}^{1-\theta}%
\norm{ f}_{\dot{B}^{-s}_{2,\infty}}^\theta, \hbox{ where }\theta=\frac{1}{\ell+1+s}.
\end{equation*}
\end{lemma}
\begin{proof}
See Lemma 4.5 in \cite{SS}.
\end{proof}

\subsection{Energy estimates}
In this subsection, we will derive the basic energy estimates for the solution to the Euler-Maxwell system \eqref{yi}. We begin with the standard energy estimates.
\begin{lemma}\label{u}
For any integer $k\ge0$, we have
\begin{eqnarray}\label{energy 1}
&&\frac{d}{dt}\sum_{l=k}^{k+2}\norm{\na^{l} ( n, u, \theta, E, B )}_{L^2}^2 +\la\sum_{l=k}^{k+2}\norm{\na^{l} (u,\theta)}_{L^2}^2\nonumber \\
&&\quad\lesssim C_kF(n,u,\theta,B)
\left(\sum_{l=k}^{k+2}\norm{\na^{l} (n, u,\theta )}_{L^2}^2+\sum_{l=k}^{k+1}\norm{\na^{l}E}_{L^2}^2+\norm{\na^{k+1} B}_{L^2}^2\right)\nonumber
\\&&\qquad+\norm{(n,u)}_{L^\infty}\norm{\na^{k+2} (n, u )}_{L^2} \norm{\na^{k+2} ( E,  B )}_{L^2},
\end{eqnarray}
where $F(n,u,\theta,B)$ is defined by
$$F(n,u,\theta,B):=\norm{ (n, u,\theta)}_{H^{\frac{k}{2}+2}\cap H^3}+\norm{ (n, u,\theta)}_{H^3}^2+\norm{\na B}_{L^2}.$$
\end{lemma}
\begin{proof}
Applying $\na^l$ $(l=k,k+1,k+2)$ to the first five equations in \eqref{yi} and then multiplying the resulting identities by $\frac{1+\theta}{1+n}\na^ln$, $1+n\na^lu$, $\frac32\frac{1+n}{1+\theta}\na^l\theta$, $\na^lE$, $\na^lB$ respectively, summing up and integrating over $\r3$, we obtain
\begin{eqnarray}\label{yi u}
\frac{d}{dt}&&\int\frac{1+\theta}{1+n}|\na^ln|^2+(1+n)|\na^lu|^2+\frac32\frac{1+n}{1+\theta}|\na^l\theta|^2+|\na^l(E,B)|^2+\norm{\na^l(u,\theta)}_{L^2}^2\nonumber\\
&&\le\int\left(\frac{\pa_t\theta}{1+n}-\frac{1+\theta}{(1+n)^2}\pa_tn\right)|\na^ln|^2+\pa_tn|\na^lu|^2+\frac32\left(\frac{\pa_tn}{1+\theta}-
\frac{1+n}{(1+\theta)^2}\pa_t\theta\right)|\na^l\theta|^2\nonumber\\
&&\quad-2\int\frac{1+\theta}{1+n}\na^l(u\cdot\na n)\na^ln+(1+n)\na^l(u\cdot\na u)\cdot\na^lu+\frac32\frac{1+n}{1+\theta}\na^l(u\cdot\na\theta)\na^l\theta\nonumber\\
&&\quad-2\int\frac{1+\theta}{1+n}\na^l((1+n){\rm div}u)\na^ln+(1+n)\na^l\left(\frac{1+\theta}{1+n}\na n\right)\cdot\na^lu\nonumber\\
&&\quad-2\int\frac{1+n}{1+\theta}\na^l((1+\theta){\rm div}u)\na^l\theta+(1+n)\na^l\na\theta\cdot\na^lu+\int\frac{1+n}{(1+\theta)}\na^l\theta\na^l(|u|^2)\nonumber\\
&&\quad-2\int(1+n)\na^l(u\times B)\cdot\na^lu-2\int n\na^lu\cdot\na^lE-\na^l(nu)\cdot\na^lE\nonumber\\
&&:=I_1+I_2+I_3+I_4+I_5+I_6+I_7.
\end{eqnarray}
First, by \eqref{n}, $\eqref{yi}_1$, $\eqref{yi}_3$ and Sobolev's embedding inequality, we easily obtain
\begin{equation*}
I_1\lesssim\left(1+\norm{(n,u,\theta)}_{H^3}\right)\norm{(n,u,\theta)}_{H^3}\norm{\na^l(n,u,\theta)}_{L^2}^2.
\end{equation*}

Next, we estimate the term $I_2$. We rewrite $I_2$ as
\begin{eqnarray*}
I_2&&=-2\int\frac{1+\theta}{1+n}\na^l(u\cdot\na n)\na^ln-2\int(1+n)\na^l(u\cdot\na u)\cdot\na^lu-3\int\frac{1+n}{1+\theta}\na^l(u\cdot\na\theta)\na^l\theta\nonumber\\
&&:=I_{21}+I_{22}+I_{23}.
\end{eqnarray*}First, we estimate $I_{21}$. By the commutator notation \eqref{commuta} and \eqref{n}, we have
\begin{eqnarray*}
I_{21}&&=-2\int\frac{1+\theta}{1+n}\na^l(u\cdot\na n)\na^ln=-2\int\frac{1+\theta}{1+n}\left(u\cdot \na\na^ln+\left[\na^l,u\right]\cdot\na n\right)\na^ln\nonumber\\
&&\lesssim\left|\int u\cdot\na(\na^ln\na^ln)+\left[\na^l,u\right]\cdot\na n\na^ln\right|.
\end{eqnarray*}
By integrating by parts, we have
\begin{equation*}
\int u\cdot\na(\na^ln\na^ln)=-\int {\rm div}u\left|\na^ln\right|^2\le\norm{{\rm div}u}_{L^\infty}\norm{\na^ln}_{L^2}^2.
\end{equation*}
We employ the commutator estimate of Lemma \ref{commutator} to bound
\begin{eqnarray*}
\int \left[\na^l,u\right]\cdot\na n\na^ln&&\le C_l\left(\norm{\na u}_{L^\infty}\norm{\na^{l-1}\na n}_{L^2}+\norm{\na^lu}_{L^2}\norm{\na n}_{L^\infty}\right)\norm{\na^ln}_{L^2}\nonumber\\
&&\le C_l\norm{\na(n,u)}_{L^\infty}\norm{\na^l(n,u)}_{L^2}\norm{\na^ln}_{L^2}.
\end{eqnarray*}
Then applying the same arguments to $I_{22}$ and $I_{23}$, by Sobolev's and Cauchy's inequalities, we obtain
\begin{equation*}
I_2\le C_l\norm{(n,u,\theta)}_{H^3}\norm{\na^l(n,u,\theta)}_{L^2}^2.
\end{equation*}

We now estimate $I_3$. By the commutator notation \eqref{commuta}, we rewrite $I_3$ as
\begin{eqnarray*}
I_3&&=-2\int (1+\theta)\left(\na^l{\rm div}u\na^ln+\na^l\na n\cdot\na^lu\right)\nonumber\\
&&\quad-2\int\frac{1+\theta}{1+n}\left[\na^l,1+n\right]{\rm div}u\na^ln+(1+n)\left[\na^l,\frac{1+\theta}{1+n}\right]\na n\cdot\na^lu.
\end{eqnarray*}By integrating by parts, we obtain
\begin{eqnarray*}
-&&2\int (1+\theta)\left(\na^l{\rm div}u\na^ln+\na^l\na n\cdot\na^lu\right)=-2\int (1+\theta){\rm div}\left(\na^lu\na^ln\right)\nonumber\\
&&=2\int\na\theta\cdot\na^lu\na^ln\lesssim\norm{\na\theta}_{L^\infty}\norm{\na^lu}_{L^2}\norm{\na^ln}_{L^2}.
\end{eqnarray*}By Lemma \ref{commutator} and \ref{A2},  Sobolev's and Cauchy's inequalities, we obtain
\begin{eqnarray*}
-&&2\int\frac{1+\theta}{1+n}\left[\na^l,1+n\right]{\rm div}u\na^ln+(1+n)\left[\na^l,\frac{1+\theta}{1+n}\right]\na n\cdot\na^lu\nonumber\\
&&\lesssim C_l\norm{\na(n,u,\theta)}_{L^\infty}\norm{\na^l(n,u,\theta)}_{L^2}^2.
\end{eqnarray*}
In fact, there is a key estimate in the following.
\begin{eqnarray}\label{beauty}
\norm{\left[\na^l,\frac{1+\theta}{1+n}\right]\na n}_{L^2}&&\le C_l\norm{\na\left(\frac{1+\theta}{1+n}\right)}_{L^\infty}\norm{\na^ln}_{L^2}+C_l\norm{\na^l\left(\frac{1+\theta}{1+n}\right)}_{L^2}\norm{\na n}_{L^\infty}\nonumber\\
&&\le C_l\norm{\na(n,\theta)}_{L^\infty}\norm{\na^ln}_{L^2}+C_l\norm{\frac{1}{1+n}}_{L^\infty}\norm{\na^l\theta}_{L^2}\norm{\na n}_{L^\infty}\nonumber\\
&&\quad+C_l\norm{(1+\theta)}_{L^\infty}\norm{\na^l\left(\frac{1}{1+n}\right)}_{L^2}\norm{\na n}_{L^\infty}\nonumber\\
&&\le C_l\norm{\na(n,\theta)}_{L^\infty}\norm{\na^l(n,\theta)}_{L^2}+C_l\norm{\na n}_{L^\infty}\norm{\na^l\left(1-\frac{1}{1+n}\right)}_{L^2}\nonumber\\
&&\le C_l\norm{\na(n,\theta)}_{L^\infty}\norm{\na^l(n,\theta)}_{L^2},
\end{eqnarray}
where we have used that $1-\frac{1}{1+n}\sim n$.
Then applying the same arguments to $I_{4}$ and $I_{5}$, by Sobolev's and Cauchy's inequalities, we obtain
\begin{equation*}
\begin{split}
I_3+I_4+I_5\le C_l\norm{(n,u,\theta)}_{H^3}\norm{\na^l(n,u,\theta)}_{L^2}^2.
\end{split}
\end{equation*}

For the term $I_6$, as in \cite{TWW},  we have that for $l=k$
\begin{equation}\label{I3 k}
I_6\le C_k\left( \norm{ u}_{H^{\frac{k}{2}}}  +\norm{\na B}_{L^2}  \right) \left(\norm{\na^{k}u}_{L^2}^2 +\norm{\na^{k+1}B}_{L^2}^2 \right);
\end{equation}
for $l=k+1$
\begin{equation*}
I_6\le C_k\left( \norm{  u}_{H^{\frac{k}{2}+1}\cap H^2}  +\norm{\na B}_{L^2}  \right) \left(\norm{\na^{k+1}u}_{L^2}^2 +\norm{\na^{k+1}B}_{L^2}^2 \right);
\end{equation*}
for $l=k+2$
\begin{eqnarray*}
I_6&&\le C_k\left( \norm{  u}_{H^{\frac{k}{2}+2}\cap H^3}  +\norm{\na B}_{L^2}  \right) \left(\norm{\na^{k+2}u}_{L^2}^2 +\norm{\na^{k+1}B}_{L^2}^2 \right)
\\&&\quad+C\norm{ u}_{L^\infty} \norm{\na^{k+2}B}_{L^2}\norm{\na^{k+2}u}_{L^2}.
\end{eqnarray*}

Next, we estimate the last term $I_7$. By Lemma \ref{commutator}, we easily obtain for $l=k$ or $k+1$,
\begin{equation*}
\begin{split}
I_7\le C_l\norm{(n,u)}_{H^2}  \left(\norm{\na^l (n, u)}_{L^2}^2 +\norm{\na^l E}_{L^2}^2 \right);
\end{split}
\end{equation*}for $l=k+2$,
\begin{equation*}
\begin{split}
I_7\le C_k\norm{(n,u)}_{L^\infty} \norm{\na^{k+2} (n, u)}_{L^2}\norm{\na^{k+2}E}_{L^2}.
\end{split}
\end{equation*}

Consequently, plugging these estimates for $I_1\sim I_7$ into \eqref{yi u} with $l=k, k+1, k+2$, and then summing up, we deduce \eqref{energy 1} from \eqref{n}.
\end{proof}

Note that in Lemma \ref{u} we only derive the dissipative estimate of $u,\theta$. We now recover the dissipative
estimates of $ n, E$ and $B$ by constructing some interactive energy functionals in the following lemma.

\begin{lemma}\label{other di}
For any integer $k\ge0$, we have that for any small fixed $\eta>0$,
\begin{eqnarray}\label{other dissipation}
&&\frac{d}{dt}\left(\sum_{l=k}^{k+1}\int \na^lu\cdot\na\na^{l} n +\sum_{l=k}^{k+1}\int \na^{l}u\cdot\na^lE -\eta\int
\na^k E \cdot\na^{k}\na\times B \right)\nonumber
\\&&\quad+\la \left(\sum_{l=k}^{k+2}\norm{\na^{l}n}_{L^2}^2+\sum_{l=k}^{k+1}\norm{\na^{l}E}_{L^2}^2+ \norm{\na^{k+1} B}_{L^2}^2\right)\nonumber
\\&&\qquad\le
 C\sum_{l=k}^{k+2}\norm{\na^{l} (u,\theta)}_{L^2}^2+C_kG(n,u,\theta,B)\left(\sum_{l=k}^{k+2}\norm{ \na^{l}(n, u,\theta )}_{L^2}^2+\norm{\na^{k+1}B}_{L^2}^2 \right),
\end{eqnarray}
where $G(n,u,\theta,B)$ is defined by
$$G(n,u,\theta,B):=\norm{ (n, u,\theta)}_{H^{\frac k2+1}\cap H^3}^2+\norm{\na B}_{L^2}^2.$$
\end{lemma}

\begin{proof} We divide the proof into four steps.

{\it Step 1: Dissipative estimate of $n$.}

Applying $\na^l$ ($l=k,k+1$) to $\eqref{yi}_2$ and then taking the $L^2$ inner product with $\na\na^{l} n$, we obtain
\begin{eqnarray}  \label{yi n}
\int&&\partial_t\na^lu \cdot\na \na^{l} n +\int \frac{1+\theta}{1+n}|\na\na^ln|^2\nonumber\\
&&\le -\int  \na^{l}E\cdot\na \na^{l} n +C \norm{\na^{l}(u,\na\theta)}_{L^2}\norm{\na^{l+1} n}_{L^2}+\norm{\na^{l}\left(u\cdot\na
u+u\times B\right)}_{L^2}\norm{\na^{l+1}  n}_{L^2}\nonumber\\
&&\quad-\int \left[\na^l\left(\frac{1+\theta}{1+n}\na n\right)-\frac{1+\theta}{1+n}\na\na^ln\right]\cdot\na^l\na n.
\end{eqnarray}

The delicate first term on the left-hand side of \eqref{yi n}
involves $\partial_t\na^{l} u$, and the key idea is to integrate by parts
in the $t$-variable and use the continuity equation $\eqref{yi}_1$. Thus integrating by
parts for both the $t$- and $x$-variables, we obtain
\begin{eqnarray*}
\int  \na^{l} \partial_tu\cdot\na\na^l n&&=\frac{d}{dt}\int
\na^{l}u\cdot\na\na^l n -\int  \na^{l}
 u\cdot \na \na^{l}\partial_t n
 =\frac{d}{dt}\int
\na^{l}u\cdot\na\na^l n +\int  \na^{l}
{\rm div} u\na^{l}\partial_t n  \\
&&=\frac{d}{dt}\int
\na^{l}u\cdot\na\na^l n -\norm{\na^{l} {{\rm div} }u}_{L^2}^2-\int \na^{l} {\rm div} u\na^{l}\left(
u \cdot \na n+ n{\rm div} u\right)
\\
&&\ge \frac{d}{dt}\int
\na^{l}u\cdot \na\na^l n -C\norm{\na^{l+1} u}_{L^2}^2
 -C\norm{\na^{l}(u \cdot \na n)}_{L^2}^2-C\norm{\na^{l}( n {\rm div} u)}_{L^2}^2.
\end{eqnarray*}
Using the commutator estimate of Lemma \ref{commutator}, we have
\begin{eqnarray}\label{san n}
\norm{\na^{l}(u \cdot \na n)}_{L^2}&&\le \norm{ u \cdot \na^{l} \na n }_{L^2}+\norm{\left[\na^l ,u \right]\cdot \na n }_{L^2}\nonumber
\\&&\le  \norm{ u}_{L^\infty} \norm{ \na^{l+1}   n }_{L^2}+C_l\norm{\na u}_{L^\infty}
\norm{\na^{l}n}_{L^2}+C_l\norm{\na^l u}_{L^2}\norm{\na n}_{L^\infty}\nonumber
\\&&\le C_l \norm{(n, u)}_{H^3} \left(\norm{ \na^{l}  ( n,u) }_{L^2}+\norm{ \na^{l+1}   n }_{L^2}\right).
\end{eqnarray}
Similarly,
\begin{equation}\label{si n}
\norm{\na^{l}( n {\rm div} u)}_{L^2}\le C_l \norm{(n, u)}_{H^3} \left(\norm{ \na^{l}  ( n,u) }_{L^2}+\norm{ \na^{l+1}  u }_{L^2}\right).
\end{equation}
Hence, we obtain
\begin{eqnarray}  \label{wu n}
 \int  \na^{l} \partial_tu\cdot\na\na^l n
&&\ge \frac{d}{dt}\int
\na^{l}u\cdot\na^l\na n -C\norm{\na^{l+1} u}_{L^2}^2\nonumber
\\&&\quad-C_l \norm{(n, u)}_{H^3}^2  \left(\norm{ \na^{l}  ( n,u) }_{L^2}^2+\norm{ \na^{l+1}  (n,u) }_{L^2}^2\right).
\end{eqnarray}

Next, integrating by parts and using the equation $\eqref{yi}_6$, we have
\begin{equation}
\begin{split}
-\int  \na^{l}E\cdot \na\na^{l} n &=\int  \na^{l}{\rm div} E\na^{l} n
=-\norm{\na^{l}n}_{L^2}^2.
\end{split}
\end{equation}
And as in \eqref{san n}--\eqref{si n}, we have
\begin{equation} \label{liu n}
\begin{split}
\norm{\na^{l}\left(u\cdot\na
u\right)}_{L^2}\le C_l \norm{u}_{H^3}  \left(\norm{ \na^{l}u }_{L^2}+\norm{ \na^{l+1} u }_{L^2}\right).
\end{split}
\end{equation}
From the estimate of $I_6$ in Lemma \ref{u}, we have that for $l=k$ or $k+1$,
\begin{equation} \label{qi n}
\norm{\na^{l}\left(u\times B\right)}_{L^2} \le C_k\left( \norm{  u}_{H^{\frac{k}{2}+1}\cap H^2}  +\norm{\na B}_{L^2}  \right) \left(\norm{\na^l u}_{L^2} +\norm{\na^{k+1}B}_{L^2} \right).
\end{equation}

Lastly, by Lemma \ref{commutator} and \eqref{beauty}, we obtain
\begin{eqnarray} \label{ba n}
-&&\int \left[\na^l\left(\frac{1+\theta}{1+n}\na n\right)-\frac{1+\theta}{1+n}\na\na^ln\right]\cdot\na^l\na n\nonumber\\
&&=-\int \left[\na^l,\frac{1+\theta}{1+n}\right]\na n\cdot\na^l\na n\lesssim C_l\norm{\na(n,\theta)}_{L^\infty}\left(\norm{\na^l(n,\theta)}_{L^2}^2+\norm{\na^{l+1}n}_{L^2}^2\right).
\end{eqnarray}

Plugging these estimates \eqref{wu n}--\eqref{ba n} into
\eqref{yi n}, by Cauchy's inequality and \eqref{n}, we obtain
\begin{eqnarray}  \label{n estimate}
\frac{d}{dt}&&\sum_{l=k}^{k+1}\int \na^lu\cdot\na\na^{l} n +\la \sum_{l=k}^{k+2}\norm{\na^{l}n}_{L^2}^2
\le
 C\sum_{l=k}^{k+2}\norm{\na^{l} (u,\theta)}_{L^2}^2\nonumber\\ &&\quad
 +C_k  \left(\norm{ (n, u, \theta)}_{H^{\frac{k}{2}+1}\cap H^3}^2+\norm{\na B}_{L^2}^2  \right)  \left(\sum_{l=k}^{k+2}\norm{ \na^{l}(n, u, \theta )}_{L^2}^2+\norm{\na^{k+1}B}_{L^2}^2 \right).
\end{eqnarray}
This completes the dissipative estimate for $n$.

{\it Step 2: Dissipative estimate of $E$.}

Applying $\na^l$ ($l=k,k+1$) to $\eqref{yi}_2$ and then taking the $L^2$ inner product with $\na^{l} E$, we obtain
\begin{eqnarray}  \label{yi E}
\int  \na^l \partial_tu \cdot\na^{l}E + \norm{\na^{l} E}_{L^2}^2
&&\le -\int \na^{l}\left(\frac{1+\theta}{1+n}\na n\right)\cdot \na^{l}E +C \norm{\na^{l}(u,\na \theta)}_{L^2}\norm{\na^{l}  E}_{L^2}\nonumber\\
 &&\quad +\norm{\na^{l}\left(u\cdot\na
u+u\times B\right)}_{L^2}\norm{\na^{l } E}_{L^2}.
\end{eqnarray}

Again, the delicate first term on the left-hand side of \eqref{yi E}
involves $\partial_t\na^{l} u$, and the key idea is to integrate by parts
in the $t$-variable and use the equation $\eqref{yi}_4$ in the Maxwell system. Thus we obtain
\begin{eqnarray}  \label{er E}
 \int  \na^{l} \partial_tu\cdot\na^lE &&=\frac{d}{dt}\int
\na^{l}u\cdot\na^lE -\int  \na^{l}
u\cdot\na^{l}\partial_t E  \nonumber\\
&&=\frac{d}{dt}\int
\na^{l}u\cdot\na^lE -\norm{\na^{l}u}_{L^2}^2-\int \na^{l} u\cdot\na^{l}\left(nu+
\na\times B\right) .
\end{eqnarray}
By Lemma \ref{commutator}, we have
\begin{equation}\label{san E}
\norm{\na^{l}\left(nu\right)}_{L^2}\le C_l  \norm{(n,u)}_{H^2}\norm{\na^l (n,u)}_{L^2} .
\end{equation}
We must be much more careful about the remaining term in \eqref{er E} since there is no small factor in front of it. The key is to use Cauchy's inequality and distinct the cases of $l=k$ and $l=k+1$ due to the weakest dissipative estimate of $B$. For $l=k$, we have
\begin{equation}\label{si E}
-\int \na^{k} u\cdot
\na\times \na^{k} B   \le \varepsilon \norm{ \na^{k+1} B}_{L^2}^2+C_\varepsilon \norm{\na^{k} u}_{L^2}^2;
\end{equation}
for $l=k+1$, integrating by parts, we obtain
\begin{eqnarray}\label{wu E}
-\int \na^{k+1} u\cdot
\na\times \na^{k+1} B  &&=
-\int \na\times \na^{k+1} u\cdot
 \na^{k+1} B\nonumber\\
&&\le \varepsilon \norm{ \na^{k+1} B}_{L^2}^2+C_\varepsilon \norm{\na^{k+2} u}_{L^2}^2.
\end{eqnarray}
By Lemma \ref{commutator}, $\eqref{yi}_6$, \eqref{n} and \eqref{beauty}, we have
\begin{eqnarray}\label{liu E}
-\int \na^{l}\left(\frac{1+\theta}{1+n}\na n\right)\cdot \na^{l}E&&=-\int\left(\left[\na^l,\frac{1+\theta}{1+n}\right]\na n+\frac{1+\theta}{1+n}\na\na^ln\right)\cdot\na^lE\nonumber\\
&&=-\int\left[\na^l,\frac{1+\theta}{1+n}\right]\na n\cdot\na^lE\nonumber\\
&&\quad+\int \frac{1+\theta}{1+n}\na^ln\cdot\na^l{\rm div}E+\int \na\left(\frac{1+\theta}{1+n}\right)\na^ln\cdot\na^lE\nonumber\\
&&\le \varepsilon \norm{ \na^lE}_{L^2}^2+ C_{l,\varepsilon}\norm{(n,\theta)}_{H^3}\norm{\na^{l}(n,\theta)}_{L^2}^2.
\end{eqnarray}

Plugging the estimates \eqref{er E}--\eqref{liu E} and \eqref{liu n}--\eqref{qi n} from Step 1 into
\eqref{yi E}, by Cauchy's inequality, we then obtain
\begin{eqnarray}  \label{E estimate}
&&\frac{d}{dt}\sum_{l=k}^{k+1}\int \na^{l}u\cdot\na^lE +\la \sum_{l=k}^{k+1}\norm{\na^{l}E}_{L^2}^2\le \varepsilon \norm{ \na^{k+1} B}_{L^2}^2+
 C_\varepsilon \sum_{l=k}^{k+2}\norm{\na^{l} (u,\theta)}_{L^2}^2\nonumber
  \\ &&\qquad +C_k  \left(\norm{ (n, u,\theta)}_{H^{\frac k2+1}\cap H^3}^2+\norm{\na B}_{L^2}^2  \right)  \left(\sum_{l=k}^{k+2}\norm{ \na^{l}(n, u,\theta )}_{L^2}^2+\norm{\na^{k+1}B}_{L^2}^2 \right).
\end{eqnarray}
This completes the dissipative estimate for $E$.

{\it Step 3: Dissipative estimate of $B$.}

Applying $\na^k$ to $\eqref{yi}_4$ and then taking the $L^2$ inner product with $-\na\times\na^{k} B$,  we obtain
\begin{eqnarray}  \label{D12}
&&-\int \na^k \partial_tE \cdot\na\times\na^{k} B +\norm{\na\times\na^{k} B}_{L^2}^2\nonumber
\\&&\quad\le  \norm{\na^{k}u}_{L^2}\norm{\na\times\na^{k} B}_{L^2}+
 \norm{\na^{k}(nu)}_{L^2}\norm{\na\times\na^{k} B}_{L^2}.
\end{eqnarray}

Integrating by parts for both the $t$- and $x$-variables and using the equation $\eqref{yi}_4$, we have
\begin{eqnarray}  \label{inter 31}
-\int \na^k \partial_tE \cdot\na\times\na^{k} B
&&=-\frac{d}{dt}\int
\na^k E \cdot\na\times\na^{k} B +\int  \na\times\na^{k}
E\cdot\na^{k}\partial_t B \nonumber \\
&&=-\frac{d}{dt}\int
\na^k E \cdot\na\times\na^{k} B - \norm{\na\times\na^{k}
E}_{L^2}^2.
\end{eqnarray}

Plugging the estimates \eqref{san E} with $l=k$ and \eqref{inter 31} into \eqref{D12} and by Cauchy's inequality, since ${\rm div} B=0$, we
then obtain
\begin{eqnarray}   \label{B estimate}
&&-\frac{d}{dt}\int
\na^k E \cdot\na^{k}\na\times B +\la \norm{\na^{k+1} B}_{L^2}^2\nonumber
\\&&\quad\le C\norm{\na^{k}u}_{L^2}^2+C\norm{\na^{k+1}
  E}_{L^2}^2+
C_k  \norm{(n,u)}_{H^2}^2 \norm{\na^k (n,u)}_{L^2}^2.
\end{eqnarray}
This completes the dissipative estimate for $B$.

{\it Step 4: Conclusion.}

Multiplying \eqref{B estimate} by a small enough but fixed constant $\eta$ and then adding it with \eqref{E estimate}  so that the second term on the right-hand side of \eqref{B estimate} can be absorbed, then choosing $\varepsilon$ small enough so that the first term on the right-hand side of \eqref{E estimate} can be absorbed; we obtain
\begin{eqnarray*}
&&\frac{d}{dt}\left(\sum_{l=k}^{k+1}\int \na^{l}u\cdot\na^lE-\eta\int
\na^k E \cdot\na^{k}\na\times B \right)+\la \left(\sum_{l=k}^{k+1}\norm{\na^{l}E}_{L^2}^2+ \norm{\na^{k+1} B}_{L^2}^2\right)
\\&&\quad\le  C \sum_{l=k}^{k+2}\norm{\na^{l} (u,\theta)}_{L^2}^2+C_kG(n,u,\theta,B)\left(\sum_{l=k}^{k+2}\norm{ \na^{l}(n, u,\theta )}_{L^2}^2+\norm{\na^{k+1}B}_{L^2}^2 \right).
\end{eqnarray*}
Here $G(n,u,\theta,B)$ is well-defined. Adding the inequality above with \eqref{n estimate}, we get \eqref{other dissipation}.
\end{proof}

\subsection{Negative Sobolev estimates}

In this subsection, we will derive the evolution of the negative Sobolev
norms of $(u,\theta,E,B)$. In order to estimate the nonlinear terms, we need to
restrict ourselves to that $s\in (0,3/2)$. We will establish the following lemma.

\begin{lemma}
\label{1Esle} For $s\in(0,1/2]$, we have
\begin{eqnarray}  \label{1E_s}
\frac{d}{dt}&&\norm{(u,\theta,E,B)}_{\dot{H}^{-s}}^2+\la\norm{(u,\theta)}_{\dot{H}^{-s}}^2\nonumber\\
&&\lesssim \left(\norm{(n,u,\theta)}_{H^2}^2+\norm{\na
B}_{H^1}^2 \right)\norm{(u,\theta,E,B)}_{\dot{H}^{-s}}+\norm{E}_{H^2}^2;
\end{eqnarray}
and for $s\in(1/2,3/2)$, we have
\begin{eqnarray}  \label{1E_s2}
\frac{d}{dt}&&\norm{(u,\theta,E,B)}_{\dot{H}^{-s}}^2+\la\norm{(u,\theta)}_{\dot{H}^{-s}}^2 \nonumber\\
&&\lesssim\left(\norm{(n,u,\theta)}_{H^{1}}^{2}
 +\norm{B}_{L^2}^{s-1/2}\norm{\na B}_{L^2}^{3/2-s}\norm{u}_{L^2}\right)\norm{(u,\theta,E,B)}_{\dot{H}^{-s}}+\norm{E}_{H^2}^2.
\end{eqnarray}
\end{lemma}

\begin{proof}
The $\Lambda^{-s}$ $(s>0)$ energy estimate of $\eqref{yi}_{2}$--$\eqref{yi}_{5}$ yield
\begin{eqnarray}  \label{1E_s_0}
&&\frac{1}{2}\frac{d}{dt}\left(\norm{(u,E,B)}_{\dot{H}^{-s}}^2+\frac{3}{2}\norm{\theta}_{\dot{H}^{-s}}^2\right) +\norm{u}_{\dot{H}^{-s}}^2+\frac32\norm{\theta}_{\dot{H}^{-s}}^2\nonumber
\\&&\quad=-\int \Lambda^{-s}\left(u\cdot\na u+\frac{\theta\na n}{1+n}+u\times B\right)
\cdot\Lambda^{-s} u -\int \Lambda^{-s}\left(\frac{\na n}{1+n}\right)
\cdot\Lambda^{-s} u\nonumber\\
&&\qquad+\int\Lambda^{-s}(nu)\cdot\Lambda^{-s} E-\frac32\int \Lambda^{-s}\left(u\cdot\na \theta+\frac23\theta{\rm div}u+\frac13|u|^2\right)\Lambda^{-s} \theta\nonumber\\
&&\quad\lesssim\norm{ u\cdot\na u+\theta\na n+u\times B}_{\dot{H}^{-s}} \norm{u}_{\dot{H}^{-s}}
+\norm{nu}_{\dot{H}^{-s}} \norm{E}_{\dot{H}^{-s}} +\norm{ \na n}_{\dot{H}^{-s}} \norm{u}_{\dot{H}^{-s}}\nonumber\\
&&\qquad+\norm{ u\cdot\na \theta+\theta{\rm div}u+|u|^2}_{\dot{H}^{-s}} \norm{\theta}_{\dot{H}^{-s}}.
\end{eqnarray}

We now restrict the value of $s$ in order to estimate the other terms on
the right-hand side of \eqref{1E_s_0}. If $s\in (0,1/2]$, then $1/2+s/3<1$
and $3/s\geq 6$. Then applying Lemma \ref{Riesz lemma}, together with H\"{o}lder's, Sobolev's and Young's inequalities, we
obtain
\begin{eqnarray*}
\norm{ u\cdot\na u }
_{\dot{H}^{-s}}
&&\lesssim \norm{ u\cdot\na u }_{L^{\frac{1}{1/2+s/3}}}
 \lesssim \norm{u}_{L^{3/s}}
\norm{\na u }_{L^{2}} \\
&&\lesssim \norm{\na u}_{L^{2}}^{1/2+s}
\norm{\na^2u }_{L^{2}}^{1/2-s}\norm{\na u}_{L^2}\lesssim \norm{ \na u}_{H^{1}}^{2}+\norm{\na u}_{L^{2}}^{2}  .
\end{eqnarray*}
Similarly, we obtain
\begin{eqnarray*}
\norm{\theta\na n+u\times B+nu+u\cdot\na \theta+\theta{\rm div}u+|u|^2}_{\dot{H}^{-s}}\lesssim\norm{(n,u,\theta)}_{H^2}^2+\norm{\na
B}_{H^1}^2.
\end{eqnarray*}

Now if $s\in (1/2,3/2)$, we shall estimate the right-hand side of
\eqref{1E_s_0} in a different way. Since $s\in
(1/2,3/2)$, we have that $1/2+s/3<1$ and $2<3/s<6$. Then applying Lemma \ref{Riesz lemma} and using
(different) Sobolev's inequality, we have
\begin{eqnarray*}
&& \nonumber
\norm{ u\cdot\na u }
_{\dot{H}^{-s}} \lesssim \norm{ u}_{L^{3/s}}\norm{\na u }
_{L^{2}} \lesssim \norm{u}_{L^{2}}^{s-1/2}
\norm{\na u }_{L^{2}}^{3/2-s}\norm{\na u}_{L^2}  \\
&&\qquad\qquad  \ \ \ \ \ \lesssim  \norm{u}_{H^{1}}^{2}+\norm{\na u}_{L^{2}}^{2};\\
&&\norm{ u\times B}_{\dot{H}^{-s}}
\lesssim \norm{B}
_{L^{2}}^{s-1/2}\norm{ \na  B}_{L^{2}}^{3/2-s}\norm{u}
_{L^{2}}.
\end{eqnarray*}Similarly, we obtain
\begin{equation*}
\norm{\theta\na n+nu+u\cdot\na \theta+\theta{\rm div}u+|u|^2}_{\dot{H}^{-s}}\lesssim\norm{(n,u,\theta)}_{H^{1}}^{2}.
\end{equation*}

Note that we   fail to estimate the remaining last term on the right-hand side of \eqref{1E_s_0} as above. To overcome this obstacle, the key point is to make full use of  the Poisson equation $\eqref{yi}_6$. Indeed, using $\eqref{yi}_6$, we have
\begin{equation*}
\begin{split}
\norm{ \na n}_{\dot{H}^{-s}}&\lesssim \norm{\Lambda^{-s}\na {\rm div} E}_{L^2}\lesssim  \norm{E}_{H^2}.
\end{split}
\end{equation*}
Now collecting all the estimates we have derived, by Cauchy's inequality, we deduce \eqref{1E_s} for $s\in(0,1/2]$ and \eqref{1E_s2} for $s\in(1/2,3/2)$.
\end{proof}

\subsection{Negative Besov estimates}

In this subsection, we will derive the evolution of the negative Besov
norms of $(u,\theta,E,B)$. The argument is similar to the previous subsection.

\begin{lemma}
\label{1Esle2} For $s\in(0,1/2]$, we have
\begin{eqnarray*}
\frac{d}{dt}\norm{(u,\theta,E,B)}_{\dot{B}_{2,\infty}^{-s}}^2+\la\norm{(u,\theta)}_{\dot{B}_{2,\infty}^{-s}}^2\lesssim \left(\norm{(n,u,\theta)}_{H^2}^2+\norm{\na
B}_{H^1}^2 \right)\norm{(u,\theta,E,B)}_{\dot{B}_{2,\infty}^{-s}}+\norm{E}_{H^2}^2;
\end{eqnarray*}
and for $s\in(1/2,3/2]$, we have
\begin{eqnarray*}
\frac{d}{dt}&&\norm{(u,\theta,E,B)}_{\dot{B}_{2,\infty}^{-s}}^2+\la\norm{(u,\theta)}_{\dot{B}_{2,\infty}^{-s}}^2 \\
&&\lesssim\left(\norm{(n,u,\theta)}_{H^{1}}^{2}
 +\norm{B}_{L^2}^{s-1/2}\norm{\na B}_{L^2}^{3/2-s}\norm{u}_{L^2}\right)\norm{(u,\theta,E,B)}_{\dot{B}_{2,\infty}^{-s}}+\norm{E}_{H^2}^2.
\end{eqnarray*}
\end{lemma}

\begin{proof}
The $\dot{\Delta}_{j}$ energy estimates of $\eqref{yi}_{2}$--$\eqref{yi}_{5}$ yield, with multiplication of $2^{-2sj}$ and then taking the supremum over ${j\in\mathbb{Z}}$,
\begin{eqnarray*}
\frac{1}{2}&&\frac{d}{dt}\left(\norm{(u,E,B)}_{\dot{B}_{2,\infty}^{-s}}^2  +\frac32\norm{\theta}_{\dot{B}_{2,\infty}^{-s}}^2\right)  + \norm{u}_{\dot{B}_{2,\infty}^{-s}}^2+\frac32\norm{\theta}_{\dot{B}_{2,\infty}^{-s}}^2\\
&&\lesssim\sup\limits_{j\in\mathbb{Z}}2^{-2sj}\left(  -\int  \dot{\Delta}_{j}\left(u\cdot\na u+\frac{\theta}{1+n}\na n+u\times B\right)
\cdot\dot{\Delta}_{j} u \right)\\
&&\quad+\sup\limits_{j\in\mathbb{Z}}2^{-2sj}\left(\int \dot{\Delta}_{j}(nu)\cdot\dot{\Delta}_{j} E -\int \dot{\Delta}_{j}\left(\frac{\na n}{1+n}\right)
\cdot\dot{\Delta}_{j} u\right)
\\
&&\quad+\sup\limits_{j\in\mathbb{Z}}2^{-2sj}\left(  -\frac32\int  \dot{\Delta}_{j}\left(u\cdot\na \theta+\frac{2}{3}\theta{\rm div}u+\frac13|u|^2\right)
\cdot\dot{\Delta}_{j} \theta \right)\\
&&\lesssim\norm{ u\cdot\na u+\theta\na n+u\times B}_{\dot{B}_{2,\infty}^{-s}} \norm{u}_{\dot{B}_{2,\infty}^{-s}}
+\norm{nu}_{\dot{B}_{2,\infty}^{-s}} \norm{E}_{\dot{B}_{2,\infty}^{-s}} +\norm{ \na n}_{\dot{B}_{2,\infty}^{-s}} \norm{u}_{\dot{B}_{2,\infty}^{-s}}\\
&&\quad +\norm{ u\cdot\na \theta+\theta{\rm div}u+|u|^2}_{\dot{B}_{2,\infty}^{-s}} \norm{\theta}_{\dot{B}_{2,\infty}^{-s}}.
\end{eqnarray*}
Then the proof is exactly the same as the proof of Lemma \ref{1Esle} except that we should apply Lemma \ref{Lp embedding} instead to estimate the $\dot{B}_{2,\infty}^{-s}$ norm. Note that we allow $s=3/2$.
\end{proof}

\section{Proof of theorems}\label{section3}

\subsection{Proof of Theorem \ref{existence}}

In this subsection, we will prove the unique global solution to the system \eqref{yi}, and the key point is that we only assume the $H^3$ norm of initial data is small.

{\it Step 1. Global small $\mathcal{E}_3$ solution.}

We first close the energy estimates at the $H^3$ level by assuming a priori that $ \sqrt{\mathcal{E}_3(t)}\le \delta$ is sufficiently small.
Taking $k=0,1$ in \eqref{energy 1} of Lemma \ref{u} and then summing up, we obtain
\begin{equation} \label{end 1}
\begin{split}
\frac{d}{dt}  \sum_{l=0}^3 \norm{ \na^{l}(n, u,\theta, E, B )}_{L^2}^2 +\la\sum_{l=0}^3\norm{\na^{l} (u,\theta)}_{L^2}^2 \lesssim  \sqrt{\mathcal{E}_3}\mathcal{D}_3+\sqrt{\mathcal{D}_3}\sqrt{\mathcal{D}_3}\sqrt{\mathcal{E}_3} \lesssim  \delta\mathcal{D}_3.\end{split}
\end{equation}
Taking $k=0,1$ in \eqref{other dissipation} of Lemma \ref{other di} and then summing up, we obtain
\begin{eqnarray}  \label{end 2}
&&\frac{d}{dt}\left(\sum_{l=0}^2\int \na^lu\cdot\na\na^{l} n +\sum_{l=0}^2\int \na^{l}u\cdot\na^lE  -\eta\sum_{l=0}^1\int
\na^l E \cdot\na^{l}\na\times B \right)\nonumber
\\&&\quad+\la \left(\sum_{l=0}^3\norm{\na^{l}n}_{L^2}^2+\sum_{l=0}^2\norm{\na^{l}E}_{L^2}^2+ \sum_{l=1}^2 \norm{\na^{l} B}_{L^2}^2\right)\lesssim
 \sum_{l=0}^3\norm{\na^{l} (u,\theta)}_{L^2}^2+\delta^2\mathcal{D}_3 .
\end{eqnarray}
 Since $\delta$ is small, we deduce from $\eqref{end 2}\times\varepsilon+ \eqref{end 1}$ that there exists an instant energy functional $\widetilde{\mathcal{E}}_3$ equivalent to ${\mathcal{E}}_3$ such that
\begin{equation*}
\frac{d}{dt}\widetilde{\mathcal{E}}_3+\mathcal{D}_3\le 0.
\end{equation*}
Integrating the inequality above directly in time, we obtain \eqref{energy inequality}. By a standard continuity argument, we then close the a priori estimates if we assume at initial time that $\mathcal{E}_3(0)\le \delta_0$ is sufficiently small. This concludes the unique global small $\mathcal{E}_3$ solution.

{\it Step 2. Global $\mathcal{E}_N$ solution.}

We shall prove this by an induction on $N\ge 3$. By \eqref{energy inequality}, then \eqref{energy inequality N} is valid for $N=3$. Assume \eqref{energy inequality N} holds for $N-1$ (then now $N\ge 4$). Taking $k=0,\dots,N-2$ in \eqref{energy 1} of Lemma \ref{u} and then summing up, we obtain
\begin{eqnarray} \label{end 3}
&&\frac{d}{dt}\sum_{l=0}^{N} \norm{ \na^{l}(n,u,\theta,E,B)}_{L^2}^2 +\la\sum_{l=0}^{N}\norm{\na^{l} (u,\theta)}_{L^2}^2 \nonumber\\&& \quad\le C_N \sqrt{\mathcal{D}_{N-1}}\sqrt{\mathcal{E}_N}\sqrt{\mathcal{D}_N}+C\sqrt{\mathcal{D}_{N-1}}\sqrt{\mathcal{D}_N}\sqrt{\mathcal{E}_N}
\le C_N \sqrt{\mathcal{D}_{N-1}}\sqrt{\mathcal{E}_N}\sqrt{\mathcal{D}_N}.
\end{eqnarray}
Here we have used the fact that $3\le \frac{N-2}{2}+2\le N-2+1$ since $N\ge 4$.
Note that it is important that we have put the two first factors in \eqref{energy 1} into the dissipation.

Taking  $k=0,\dots,N-2$ in \eqref{other dissipation} of Lemma \ref{other di} and then summing up, we obtain
\begin{eqnarray}  \label{end 4}
&&\frac{d}{dt}\left(\sum_{l=0}^{N-1}\int \na^lu\cdot\na\na^{l} n +\sum_{l=0}^{N-1}\int \na^{l}u\cdot\na^lE -\eta\sum_{l=0}^{N-2}\int
\na^l E \cdot\na\times\na^{l} B \right)\nonumber
\\&&\quad+\la \left(\sum_{l=0}^{N}\norm{\na^{l}n}_{L^2}^2+\sum_{l=0}^{N-1}\norm{\na^{l}E}_{L^2}^2+ \sum_{l=1}^{N-1} \norm{\na^{l} B}_{L^2}^2\right)\nonumber
 \\ &&\qquad\le
 C\sum_{l=0}^{N}\norm{\na^{l} (u,\theta)}_{L^2}^2
 +C_N  \sqrt{\mathcal{D}_{N-1}}\sqrt{\mathcal{D}_N}\sqrt{\mathcal{E}_N}.
\end{eqnarray}
We deduce from $\eqref{end 4}\times \varepsilon+\eqref{end 3}$ that there exists an instant energy functional $\widetilde{\mathcal{E}}_N$ equivalent to $\mathcal{E}_N$ such that, by Cauchy's inequality,
\begin{equation*}
\frac{d}{dt}\widetilde{\mathcal{E}}_N+\mathcal{D}_N\le  C_N \sqrt{\mathcal{D}_{N-1}}\sqrt{\mathcal{E}_N}\sqrt{\mathcal{D}_N}
 \le \varepsilon \mathcal{D}_N+ C_{N,\varepsilon} {\mathcal{D}_{N-1}} {\mathcal{E}_N}.
\end{equation*}
This implies
\begin{equation*}
\frac{d}{dt}\widetilde{\mathcal{E}}_N+\frac{1}{2}\mathcal{D}_N
 \le  C_{N } {\mathcal{D}_{N-1}} {\mathcal{E}_N}.
\end{equation*}
We then use the standard Gronwall lemma and the induction hypothesis to deduce that
\begin{eqnarray*}
 \mathcal{E}_N(t)+\int_0^t\mathcal{D}_N(\tau)\,d\tau
&& \le  C\mathcal{E}_N(0)e^{C_{N }\int_0^t{\mathcal{D}_{N-1}} (\tau)\,d\tau}\le  C\mathcal{E}_N(0)e^{C_{N }P_{N-1}\left(\mathcal{E}_{N-1}(0)\right)}\\&&\le  C\mathcal{E}_N(0)e^{C_{N }P_{N-1}\left(\mathcal{E}_{N}(0)\right)}\equiv P_N\left(\mathcal{E}_{N}(0)\right).
\end{eqnarray*}
This concludes the global $\mathcal{E}_N$ solution. The proof of Theorem \ref{existence} is completed.\hfill$\Box$

\subsection{Proof of Theorem \ref{decay}}
In this subsection, we will prove the various time decay rates of the unique global solution to the system \eqref{yi} obtained in Theorem \ref{existence}. Fix $N\ge 5$. We need to assume that $\mathcal{E}_N(0)\le \delta_0=\delta_0(N)$ is small. Then Theorem \ref{existence} implies that there exists a unique global $\mathcal{E}_N$ solution, and $\mathcal{E}_N(t)\le P_{N}\left(\mathcal{E}_N(0)\right) \le \delta_0$ is small for all time $t$. Since now our $\delta_0$ is relative small with respect to $N$, we just ignore the $N$ dependence of the constants in the energy estimates in the previous section.

{\it Step 1. The $\dot{H}^{-s}$ or $\dot{B}_{2,\infty}^{-s}$ norm is preserved along time evolution.}

\eqref{H-sbound} and \eqref{H-sbound Besov} indicate that the $\dot{H}^{-s}$ or $\dot{B}_{2,\infty}^{-s}$ norm of $(u,\theta,E,B)(t)$ is preserved along time evolution.
First, we prove \eqref{H-sbound} by Lemma \ref{1Esle}. However, we are not
able to prove them for all $s\in[0,3/2)$ at this moment. We must distinct the arguments by the value of $s$.
First, for $s\in (0,1/2]$, integrating \eqref{1E_s} in time, by \eqref{energy inequality} we obtain that for $s\in (0,1/2]$,
\begin{eqnarray*}
\norm{(u,\theta,E,B)(t)}_{\dot{H}^{-s}}^2&&\lesssim \norm{(u_0,\theta_0,E_0,B_0) }_{\dot{H}^{-s}}^2+ \int_{0}^{t} \mathcal{D}_3(\tau ) \left(1+\norm{(u,\theta,E,B)(\tau)}_{\dot{H}^{-s}}\right)
 \,d\tau  \nonumber\\
&&\leq  C_0\left(1+\sup_{0\leq \tau \leq t}\norm{(u,\theta,E,B)(\tau)}_{\dot{H}^{-s}}\right).
\end{eqnarray*}
By Cauchy's inequality, this together with \eqref{energy inequality}
gives \eqref{H-sbound} for $s\in [ 0,1/2]$ and thus verifies \eqref{basic decay} for $s\in [ 0,1/2]$. Next, we let $s\in (1/2,1)$.
Observing that we have $ (u_0,\theta_0,E_0,B_0)\in \dot{H}^{-1/2}
$ since $\dot{H}^{-s}\cap L^2\subset\dot{H}^{-s^{\prime}}$ for any $
s^{\prime}\in [0,s]$, we then deduce from what we have proved for
\eqref{basic decay} with $s=1/2$ that the following decay
result holds:
\begin{equation}  \label{1proof14}
\norm{\na^k  (n,u,\theta,E,B)(t)}
_{H^{2}}  \le C_0(1+t)^{-\frac{k+  {1}/{2}}{2}}\ \hbox{ for }k=0,1.
\end{equation}
Here, since we have required $N\ge 5$ and now $s=1/2$, so we can have taken $k=1$ in \eqref{basic decay}. Thus by \eqref{1proof14}, \eqref{energy inequality} and H\"older's inequality, we deduce from \eqref{1E_s2} that for $
s\in(1/2,1)$,
\begin{eqnarray}  \label{1-sin2''}
\norm{(u,\theta,E,B)(t)}_{\dot{H}^{-s}}^2
&&\lesssim \norm{(u_0,\theta_0,E_0,B_0) }_{\dot{H}^{-s}}^2+ \int_{0}^{t} \mathcal{D}_3(\tau ) \left(1+\norm{(u,\theta,E,B)(\tau)}_{\dot{H}^{-s}}\right)
 \,d\tau\nonumber\\
&&\quad+\int_0^t\norm{B(\tau )}_{L^2}^{s-1/2}\norm{\na B(\tau )}_{L^2}^{3/2-s}\sqrt{\mathcal{D}_3(\tau )}\norm{(u,\theta,E,B)(\tau)}_{\dot{H}^{-s}}\,d\tau\nonumber\\
&&\le C_0\left(1+\left(1+\int_0^t(1+\tau)^{-2(1-s/2)}\,d\tau \right)\sup_{0\le\tau\le t}\norm{(u,\theta,E,B)(\tau)}_{\dot{H}^{-s}}\right)\nonumber\\
&&\le C_0\left(1+\sup_{0\le\tau\le t}\norm{(u,\theta,E,B)(\tau)}_{\dot{H}^{-s}}\right).
\end{eqnarray}
Here we have used the fact $s\in(1/2,1)$ so that the time integral in \eqref{1-sin2''} is finite.
This gives \eqref{H-sbound} for $s\in (1/2,1)$ and thus verifies \eqref{basic decay} for $s\in (1/2,1)$.
Now let $s\in [1,3/2)$. We choose $s_0$ such that $s-1/2<s_0<1$. Hence, $ (u _0,\theta_0,E_0,B_0)\in \dot{H}^{-s_0}$. We then deduce from what we have proved for \eqref{basic decay} with $s=s_0$ that the following decay
result holds:
\begin{equation}  \label{2proof14}
\norm{\na^k  (n,u,\theta,E,B)(t)}
_{H^{2}}  \le C_0(1+t)^{- \frac{k+s_0}{2} }\ \hbox{ for }k=0,1.
\end{equation}
Here, since we have required $N\ge 5$ and now $s=s_0<1$, so we can have taken $k=1$ in \eqref{basic decay}. Thus by \eqref{2proof14} and H\"older's inequality, we deduce from \eqref{1E_s2} that for $
s\in[1,3/2)$, similarly as in \eqref{1-sin2''},
\begin{eqnarray}\label{mmm}
 \norm{(u,\theta,E,B)(t)}_{\dot{H}^{-s}}^2
&&\le C_0\left(1+\left(1+\int_0^t(1+\tau)^{- ( {s_0} + 3/2-  s )}\,d\tau \right)\sup_{0\le\tau\le t}\norm{(u,\theta,E,B)(\tau)}_{\dot{H}^{-s}}\right)\nonumber\\
&&\le C_0\left(1+\sup_{0\le\tau\le t}\norm{(u,\theta,E,B)(\tau)}_{\dot{H}^{-s}}\right).
\end{eqnarray}
Here we have used the fact $s-s_0<1/2$ so that the time integral in \eqref{mmm} is finite.
This gives \eqref{H-sbound} for $s\in [1,3/2)$ and thus verifies \eqref{basic decay} for $s\in [1,3/2)$. Note that
\eqref{H-sbound Besov} can be proved similarly except that we use instead  Lemma \ref{1Esle2}.

{\it Step 2. Basic decay.}

For the convenience of presentations, we define a family of energy functionals and the corresponding dissipation rates with {\it minimum derivative counts} as
\begin{equation}\label{1111}
\mathcal{E}_{k}^{k+2}=\sum_{l=k}^{k+2}\norm{ \na^{l}(n, u,\theta, E, B )}_{L^2}^2
\end{equation}
and
\begin{equation}\label{2222}
\mathcal{D}_{k}^{k+2}=\sum_{l=k}^{k+2}\norm{\na^{l}(n,u,\theta)}_{L^2}^2+\sum_{l=k}^{k+1}\norm{\na^{l}E}_{L^2}^2+ \norm{\na^{k+1} B}_{L^2}^2.
\end{equation}

By Lemma \ref{u}, we have that for $k=0,\dots,N-2$,
\begin{eqnarray}  \label{end 5}
&&\frac{d}{dt}\sum_{l=k}^{k+2}\norm{ \na^{l}(n, u,\theta, E, B )}_{L^2}^2 +\la\sum_{l=k}^{k+2}\norm{\na^{l} (u,\theta)}_{L^2}^2\nonumber \\
&&\quad\lesssim \sqrt{\delta_0} \mathcal{D}_k^{k+2} +\norm{(n,u)}_{L^\infty}\norm{\na^{k+2}(n, u )}_{L^2}
\norm{ \na^{k+2}( E,  B )}_{L^2}.
\end{eqnarray}
By Lemma \ref{other di}, we have that for $k=0,\dots,N-2$,
\begin{eqnarray} \label{end 6}
&&\frac{d}{dt}\left(\sum_{l=k}^{k+1}\int \na^lu\cdot\na\na^{l} n +\sum_{l=k}^{k+1}\int \na^{l}u\cdot\na^lE  -\eta\int
\na^k E \cdot\na^{k}\na\times B \right)\nonumber
\\&&\quad+\la \left(\sum_{l=k}^{k+2}\norm{\na^{l}n}_{L^2}^2+\sum_{l=k}^{k+1}\norm{\na^{l}E}_{L^2}^2+ \norm{\na^{k+1} B}_{L^2}^2\right)\nonumber
\\&&\qquad\lesssim\sum_{l=k}^{k+2}\norm{\na^{l} (u,\theta)}_{L^2}^2
 +\delta_0 \sum_{l=k}^{k+2}\norm{\na^{l} (n, u )}_{L^2}^2.
\end{eqnarray}
Since $\delta_0$ is small, we deduce from $\eqref{end 6}\times \varepsilon+\eqref{end 5}$ that there exists an instant energy functional $\widetilde{\mathcal{E}}_k^{k+2}$ equivalent to $\mathcal{E}_k^{k+2}$ such that
\begin{equation}\label{energy}
\frac{d}{dt}\widetilde{\mathcal{E}}_k^{k+2}+\mathcal{D}_k^{k+2}\lesssim   \norm{(n,u)}_{L^\infty}\norm{\na^{k+2}(n, u )}_{L^2}
\norm{ \na^{k+2}( E,  B )}_{L^2}.
\end{equation}
Note that we can not absorb the right-hand side of \eqref{energy} by the dissipation $\mathcal{D}_k^{k+2}$ since it does not contain $\norm{\na^{k+2}( E, B )}_{L^2}^2$. We will distinct the arguments by the value of $k$. If $k=0$ or $k=1$, we bound $\norm{\na^{k+2}( E, B )}_{L^2}$ by the energy. Then we have that for $k=0,1$,
\begin{equation*}
\frac{d}{dt}\widetilde{\mathcal{E}}_k^{k+2}+\mathcal{D}_k^{k+2}\lesssim   \sqrt{\mathcal{D}_k^{k+2}}\sqrt{\mathcal{D}_k^{k+2}}\sqrt{{\mathcal{E}}_3}\lesssim \sqrt{\delta_0}\mathcal{D}_k^{k+2},
\end{equation*}
which implies
\begin{equation*}
\frac{d}{dt}\widetilde{\mathcal{E}}_k^{k+2}+\mathcal{D}_k^{k+2}\le 0.
\end{equation*}
If $k\ge 2$, we have to bound $\norm{\na^{k+2}( E, B )}_{L^2}$ in term of $\norm{\na^{k+1}( E, B )}_{L^2}$ since $\sqrt{\mathcal{D}_k^{k+2}}$ can not control $\norm{(n,u)}_{L^\infty}$. The key point is to use the regularity interpolation method developed in \cite{GW,SG06}. By Lemma \ref{A1}, we have
\begin{eqnarray}\label{kkk}
&&\norm{(n,u)}_{L^\infty}\norm{ \na^{k+2}(n,u )}_{L^2} \norm{ \na^{k+2}( E, B )}_{L^2}\nonumber\\&&\quad\lesssim\norm{(n,u)}_{L^2}^{1-\frac{3 }{2k}}\norm{\na^{k}(n,u )}_{L^2}^{\frac{3 }{2k}}\norm{\na^{k+2}(n,u )}_{L^2}\norm{\na^{k+1}( E, B )}_{L^2}^{1-\frac{3 }{2k}}
\norm{\na^{\al}( E, B )}_{L^2}^{\frac{3 }{2k}},
\end{eqnarray}
where $\alpha$ is defined by
\begin{equation*}
 k+2=(k+1)\times\left(1-\frac{3 }{2k}\right)+\alpha\times \frac{3 }{2k}
 \Longrightarrow \alpha=\frac{5}{3}k+1.
\end{equation*}
Hence, for $k\ge 2$, if $N\ge \frac{5}{3}k+1\Longleftrightarrow 2\le k\le \frac{3}{5}(N-1)$, then by \eqref{kkk}, we deduce from \eqref{energy} that
\begin{equation*}
\frac{d}{dt}\widetilde{\mathcal{E}}_k^{k+2}+\mathcal{D}_k^{k+2} \lesssim  \sqrt{{\mathcal{E}}_{N}} {\mathcal{D}_k^{k+2}}\lesssim \sqrt{\delta_0}\mathcal{D}_k^{k+2},
\end{equation*}
which allow us to arrive at that for any integer $k$ with $0\le k\le  \frac{3}{5}(N-1) $ (note that $N-2\ge \frac{3}{5}(N-1)\ge 2$ since $N\ge 5$), we have
\begin{equation}\label{k energy}
\frac{d}{dt}\widetilde{\mathcal{E}}_k^{k+2}+\mathcal{D}_k^{k+2} \le 0.
\end{equation}

We now begin to derive the decay rate from \eqref{k energy}. Using Lemma
\ref{1-sinte} and \eqref{H-sbound}, we have that for $s\ge 0$ and $k+s>0$,
\begin{equation*}
\norm{\na^k B }_{L^2} \le  \norm{B}_{\dot{H}^{-s}}^{\frac{1}{k+1+s}}\norm{\na^{k+1}B }_{L^2}^{\frac{k+s}{k+1+s}}
 \le  C_0\norm{\na^{k+1}B}_{L^2}^{\frac{k+s}{k+1+s}}.
\end{equation*}
Similarly, using Lemma
\ref{Besov interpolation} and \eqref{H-sbound Besov}, we  have that for $s>0$ and $k+s>0$,
\begin{equation*}
\norm{\na^k B }_{L^2} \le  \norm{B }_{\dot{B}_{2,\infty}^{-s}}^{\frac{1}{k+1+s}}\norm{\na^{k+1}B }_{L^2}^{\frac{k+s}{k+1+s}}
 \le  C_0\norm{\na^{k+1}B}_{L^2}^{\frac{k+s}{k+1+s}}.
\end{equation*}
On the other hand, for $k+2<N$, we have
\begin{equation*}
\begin{split}
\norm{ \na^{k+2}( E, B )}_{L^2}\le  \norm{\na^{k+1}( E, B )}_{L^2}^{ \frac{N-k-2}{N-k-1}}
\norm{\na^N( E, B )}_{L^2}^{\frac{1}{N-k-1}}\le  C_0\norm{\na^{k+1}( E, B )}_{L^2}^{ \frac{N-k-2}{N-k-1}}.
\end{split}
\end{equation*}
Then we deduce from \eqref{k energy} that
\begin{equation*}
\frac{d}{dt}\widetilde{\mathcal{E}}_k^{k+2}+\left\{\mathcal{E}_k^{k+2}\right\}^{1+\vartheta} \le 0,
\end{equation*}
where $\vartheta=\max\left\{\frac{1}{k+s},\frac{1}{N-k-2}\right\}$.
Solving this inequality directly, we obtain in particular that
\begin{equation}\label{nnn}
\mathcal{E}_k^{k+2}  (t) \le \left\{\left[\mathcal{E}_k^{k+2}(0)\right]^{-\vartheta}+\vartheta  t\right\}^{-  {1}/{\vartheta}}
\le C_0 (1+ t)^{- {1}/{\vartheta}}=C_0 (1+ t)^{-\min\left\{ {k+s}, {N-k-2}\right\}}.
\end{equation}
Notice that \eqref{nnn} holds also for $k+s=0$ or $k+2=N$. So, if we want to obtain the optimal decay rate of the whole solution for the spatial derivatives of order $k$, we only need to assume $N$ large enough (for fixed $k$ and $s$) so  that $k+s\le N-k-2$. Thus we should require that
\begin{equation*}
N\ge \max\left\{k+2, \frac{5}{3}k+1, 2k+2+s\right\}= 2k+2+s.
\end{equation*}
This proves the optimal decay \eqref{basic decay}.

{\it Step 3. Further decay.}

We first prove \eqref{further decay1} and \eqref{further decay11}. First, noticing that $-n={\rm div}E$, by \eqref{basic decay}, if $N\ge2k+4+s$, then
\begin{equation}\label{n further decay}
\norm{\na^kn(t)}_{L^2}\lesssim\norm{\na^{k+1}E(t)}_{L^2}\lesssim C_0(1+t)^{-\frac{k+1+s}{2}}.
\end{equation}

Next, applying $\na^k$ to $\eqref{yi}_2, \eqref{yi}_3, \eqref{yi}_4$ and then multiplying the resulting identities
by $\na^ku$, $\frac32\na^k\theta$, $\na^kE$ respectively, summing up and integrating over $\mathbb{R}^3$, we obtain
\begin{eqnarray}\label{uE yi}
&&\frac{1}{2}\frac{d}{dt}\int\norms{\na^k(u,E)}^2+\frac32\norms{\na^k\theta}^2+\norm{\na^ku}_{L^2}^2+\frac32\norm{\na^k\theta}_{L^2}^2\nonumber\\
&&\quad=-\int\na^k\left(u\cdot\na u
+\frac{1+\theta}{1+n}\na n+u\times B\right)\cdot\na^ku+\int\na^k\left( \na\times B+ nu\right)\cdot\na^kE\nonumber\\
&&\qquad-\frac32\int \na^k\left(u\cdot\na \theta+\frac23\theta{\rm div}u+\frac13|u|^2\right)\na^k \theta\nonumber
\\&&\quad\lesssim \norm{\na^{k}\left(u\cdot\na
u+\frac{1+\theta}{1+n}\na n+u\times B\right)}_{L^2}\norm{\na^{k} u}_{L^2}
+\norm{ \na^{k}\left(\na\times B+nu
\right)}_{L^2}\norm{\na^{k} E}_{L^2}\nonumber\\
&&\qquad+\norm{\na^{k}\left(u\cdot\na
\theta+\theta{\rm div}u+|u|^2\right)}_{L^2}\norm{\na^{k} \theta}_{L^2}.
\end{eqnarray}
On the other hand, taking $l=k$ in \eqref{yi E}, we may have
\begin{eqnarray} \label{uE san}
&&\int  \na^k \partial_tu \cdot\na^{k}E +\norm{\na^{k} E}_{L^2}^2\nonumber
\\
&&\quad\lesssim\left[\norm{\na^{k+1}\theta}_{L^2}+\norm{\na^{k} u}_{L^2}+\norm{\na^{k}\left(u\cdot\na
u+\frac{1+\theta}{1+n}\na n+u\times B\right)}_{L^2}\right]\norm{\na^{k} E}_{L^2}.
\end{eqnarray}
Substituting \eqref{er E} with $l=k$ into \eqref{uE san}, we may then have
\begin{eqnarray}\label{uE wu}
\frac{d}{dt}&&\int  \na^ku \cdot\na^{k}E +\norm{\na^{k} E}_{L^2}^2\nonumber
\\
 &&\lesssim \norm{\na^{k} u}_{L^2}^2 +\norm{\na^{k}\left(u\cdot\na
u+\frac{1+\theta}{1+n}\na n+u\times B\right)}_{L^2}\norm{\na^{k} E}_{L^2}\nonumber\\
&&\quad+\left(\norm{\na^{k+1}\theta}_{L^2}+ \norm{\na^{k} u}_{L^2}\right)\norm{\na^{k}  E}_{L^2}+\norm{ \na^{k}\left(\na\times B+nu
\right)}_{L^2}\norm{\na^{k} u}_{L^2}.
\end{eqnarray}
Since $\varepsilon$ is small, we deduce from $\eqref{uE wu}\times\varepsilon+\eqref{uE yi}$ that there exists $\mathcal{F}_k(t)$ equivalent to $\norm{\na^k(u,\theta,E)(t)}_{L^2}^2$ such that, by Cauchy's inequality, Lemma \ref{commutator}, \eqref{beauty}, \eqref{I3 k}, \eqref{basic decay} and \eqref{n further decay},
\begin{eqnarray}\label{uE liu}
\frac{d}{dt}\mathcal{F}_k(t)+\mathcal{F}_k(t)
&&\lesssim \norm{\na^{k+1} \theta}_{L^2}^2+\norm{\na^{k+1} B}_{L^2}^2 +\norm{\na^{k}\left(u\cdot\na
u+\frac{1+\theta}{1+n}\na n+u\times B\right)}_{L^2}^2\nonumber\\
&&\quad+\norm{\na^{k}\left(u\cdot\na \theta+\theta{\rm div}u+|u|^2\right)}_{L^2}^2 +\norm{ \na^{k}(nu)}_{L^2}^2\nonumber\\
&&\lesssim \norm{\na^{k+1} (n,\theta, B)}_{L^2}^2 + \left(\norm{u}_{H^{\frac k2}}+\norm{\na B}_{L^2}\right)^2\norm{\na^{k+1}B}_{L^2}^2\nonumber\\
&&\quad+\norm{(\na n,\na\theta,u)}_{L^\infty}^2\norm{\na^kn}_{L^2}^2+\norm{(n,u,\theta)}_{L^\infty}^2\norm{\na^{k+1}(n,u,\theta)}_{L^2}^2\nonumber\\
&&\le C_0(1+t)^{-(k+1+s)},
\end{eqnarray}
where we required $N\ge2k+4+s$.
Applying the standard Gronwall lemma to \eqref{uE liu}, we obtain
\begin{equation*}
\begin{split}
\mathcal{F}_k(t)\le \mathcal{F}_k(0)e^{-t}+C_0\int_0^te^{-(t-\tau)}(1+\tau)^{-(k+1+s)}\,d\tau\lesssim C_0(1+t)^{-(k+1+s)}.
\end{split}
\end{equation*}
This implies
\begin{equation*}
\norm{\na^k(u,\theta,E)(t)}_{L^2}\lesssim \sqrt{\mathcal{F}_k(t)}\lesssim C_0(1+t)^{-\frac{k+1+s}{2}}.
\end{equation*}
We thus complete the proof of \eqref{further decay1}. Notice that \eqref{further decay11} now follows by \eqref{n further decay} with the improved decay rate of $E$ in \eqref{further decay1}, just requiring $N\ge2k+6+s$.

Now we prove \eqref{further decay2}. Assuming $ B_\infty =0$, then we can extract the following system from  $\eqref{yi}_1$--$\eqref{yi}_2$, denoting $\psi={\rm div} u$,
\begin{equation}\label{npsi yi}
\left\{
\begin{array}{lll}
\displaystyle\partial_tn+\psi=-u\cdot\na n
-n{\rm div} u,   \\
\displaystyle\partial_t \psi+ \psi-n=-\Delta \theta-{\rm div}\left(u\cdot\na u
+\frac{1+\theta}{1+n}\na n+u\times B\right).
\end{array}
\right.
\end{equation}
Applying $\na^k$ to $\eqref{npsi yi}$ and then multiplying the resulting identities
by $\na^kn$, $\na^k\psi$, respectively, summing up and integrating over $\mathbb{R}^3$, we obtain
\begin{eqnarray}\label{npsi er}
\frac{1}{2}\frac{d}{dt}\int \norms{\na^k n }^2+\norms{\na^k \psi }^2+&&\norm{\na^k\psi}_{L^2}^2=-\int\na^k(u\cdot\na n+n{\rm div} u) \na^kn-\int\na^k \Delta \theta \na^k\psi\nonumber\\
&&\qquad\quad-\int\na^k {\rm div}\left(u\cdot\na u
+\frac{1+\theta}{1+n}\na n+u\times B\right) \na^k\psi.
\end{eqnarray}
Applying $\na^k$ to $\eqref{npsi yi}_2$ and then multiplying by $-\na^kn$, as before integrating by parts over $t$ and $x$ variables and using the equation $\eqref{npsi yi}_1$, we may obtain
\begin{eqnarray}\label{npsi san}
-\frac{d}{dt}\int\na^k\psi \na^kn+&&\norm{\na^kn}_{L^2}^2
=\norm{\na^k\psi}_{L^2}^2+\int\na^kn \na^k\psi+\int\na^k(u\cdot\na n+n{\rm div} u) \na^k\psi\nonumber\\
&&\qquad\quad+\int\na^k\left[\Delta \theta+{\rm div}\left(u\cdot\na u
+\frac{1+\theta}{1+n}\na n+u\times B\right)\right] \na^kn.
\end{eqnarray}

Since $\varepsilon$ is small, we deduce from $\eqref{npsi san}\times\varepsilon+\eqref{npsi er}$ that there exists $\mathcal{G}_k(t)$ equivalent to $\norm{\na^k(n,\psi)}_{L^2}^2$ such that, by Cauchy's inequality,
\begin{eqnarray}\label{npsi si}
 \frac{d}{dt}\mathcal{G}_k(t)+\mathcal{G}_k(t)&&\lesssim  \norm{\na^{k+2}\theta}_{L^2}^2+\norm{\na^{k+1}(u\cdot\na u)}_{L^2}^2+\norm{\na^{k+1}\left(\frac{1+\theta}{1+n}\na n\right)}_{L^2}^2 \nonumber\\
 &&\quad +\norm{\na^{k+1}(u\times B)}_{L^2}^2+ \norm{\na^k(u\cdot\na n)}_{L^2}^2+\norm{\na^k(n{\rm div}u)}_{L^2}^2.
\end{eqnarray}
By Lemma \ref{commutator}, \eqref{beauty} and Cauchy's inequality, we obtain
\begin{eqnarray*}
\norm{\na^{k+1}\left(\frac{1+\theta}{1+n}\na n\right)}_{L^2}^2
&&\lesssim\norm{\left[\na^{k+1},\frac{1+\theta}{1+n}\right]\na n}_{L^2}^2+\norm{\frac{1+\theta}{1+n}\na^{k+2}n}_{L^2}^2\nonumber\\
&&\lesssim\norm{\na^{k+2}n}_{L^2}^2+\norm{\na(n,\theta)}_{L^\infty}^2\norm{\na^{k+1}(n,\theta)}_{L^2}^2,\\
\end{eqnarray*}and
\begin{eqnarray*}
\norm{\na^{k+1}(u\times B)}_{L^2}^2&&=\norm{u\times \na^{k+1}B+\left[\na^{k+1}, u\right]\times B}_{L^2}^2\norm{u\times \na^{k+1}B}_{L^2}^2+\norm{\left[\na^{k+1}, u\right]\times B}_{L^2}^2\nonumber\\
&&\lesssim\norm{u}_{L^\infty}^2\norm{\na^{k+1}B}_{L^2}^2+\norm{\na u}_{L^\infty}^2\norm{\na^{k}B}_{L^2}^2+\norm{\na^{k+1}u}_{L^2}^2\norm{B}_{L^\infty}^2.
\end{eqnarray*}
The other nonlinear terms on the right-hand side of \eqref{npsi si} can be estimated similarly.
Hence, we deduce from \eqref{npsi si} that, by \eqref{basic decay}--\eqref{further decay11},
\begin{eqnarray}\label{npsi qi}
&&\frac{d}{dt}\mathcal{G}_k(t)+\mathcal{G}_k(t)\nonumber
 \\&&\quad\lesssim \norm{\na^{k+2}(n,\theta)}_{L^2}^2+\norm{u}_{L^\infty}^2\norm{\na^{k+1}B}_{L^2}^2+\norm{\na u}_{L^\infty}^2\norm{\na^{k}B}_{L^2}^2+\norm{B}_{L^\infty}^2\norm{\na^{k+1}u}_{L^2}^2\nonumber\\
&&\qquad+\norm{(n,u)}_{L^\infty}^2\norm{\na^{k+2}(n,u)}_{L^2}^2 +\norm{\na(n,u,\theta)}_{L^\infty}^2\norm{\na^{k+1}(n,u,\theta)}_{L^2}^2\nonumber\\
&&\quad\le C_0\left((1+t)^{-(k+3+s)}+  (1+t)^{-(k+7/2+2s)} + (1+t)^{-(k+11/2+2s)} \right)\nonumber
\\&&\quad\le C_0(1+t)^{-(k+3+s)},
\end{eqnarray}
where we required $N\ge2k+8+s$.
Applying the Gronwall lemma to \eqref{npsi qi} again, we obtain
\begin{equation*}
\mathcal{G}_k(t)\le \mathcal{G}_k(0)e^{-t}+C_0\int_0^te^{-(t-\tau)}(1+\tau)^{-(k+3+s)}\,d\tau\le  C_0(1+t)^{-(k+3+s)}.
\end{equation*}
This implies
\begin{equation}\label{npsi ba}
\norm{\na^k(n,\psi)(t)}_{L^2}\lesssim \sqrt{\mathcal{G}_k(t) }\le C_0 (1+t)^{-\frac{k+3+s}{2}}.
\end{equation}

Now we consider the following system which consists of \eqref{npsi yi} and $\eqref{yi}_3$:
\begin{equation}\label{npsitheta yi}
\left\{
\begin{array}{lll}
\displaystyle\partial_tn+\psi=-u\cdot\na n
-n{\rm div} u,   \\
\displaystyle\partial_t \psi+ \psi-n=-\Delta \theta-{\rm div}\left(u\cdot\na u
+\frac{1+\theta}{1+n}\na n+u\times B\right),\\
\displaystyle\partial_t \theta+ \theta=-u\cdot\na\theta-\frac23(1+\theta){\rm div}u+\frac13|u|^2.
\end{array}
\right.
\end{equation}
First, we have the standard energy identity for the system \eqref{npsitheta yi}
\begin{eqnarray}\label{npsitheta er}
\frac{1}{2}&&\frac{d}{dt}\int \norms{\na^k n }^2+\norms{\na^k \psi }^2+\norms{\na^k \theta }^2+\norm{\na^k(\psi,\theta)}_{L^2}^2\nonumber\\
&&=-\int\na^k(u\cdot\na n+n{\rm div} u) \na^kn-\int\na^k \Delta \theta \na^k\psi-\frac23\int\na^k\psi\na^k\theta\nonumber\\
&&\quad-\int\na^k {\rm div}\left(u\cdot\na u
+\frac{1+\theta}{1+n}\na n+u\times B\right) \na^k\psi\nonumber\\
&&\quad-\int\na^k \left(u\cdot\na\theta+\frac23\theta{\rm div}u-\frac13|u|^2\right) \na^k\theta.
\end{eqnarray}
Since $\varepsilon$ is small, we deduce from $\eqref{npsi san}\times\varepsilon+\eqref{npsitheta er}$ that there exists $\mathcal{H}_k(t)$ equivalent to $\norm{\na^k(n,\psi,\theta)}_{L^2}^2$ such that, by Cauchy's inequality and \eqref{npsi ba}
\begin{eqnarray}\label{npsitheta san}
 \frac{d}{dt}\mathcal{H}_k(t)+\mathcal{H}_k(t)&&\lesssim  \norm{\na^{k+2}\theta}_{L^2}^2+\norm{\na^{k+1}(u\cdot\na u)}_{L^2}^2+\norm{\na^{k+1}\left(\frac{1+\theta}{1+n}\na n\right)}_{L^2}^2 \nonumber\\
&&\quad +\norm{\na^{k+1}(u\times B)}_{L^2}^2+ \norm{\na^k(u\cdot\na n)}_{L^2}^2+\norm{\na^k(n{\rm div}u)}_{L^2}^2\nonumber\\
&&\quad +\norm{\na^{k}\psi}_{L^2}^2+ \norm{\na^k(|u|^2)}_{L^2}^2+\norm{\na^k(u\cdot\na \theta)}_{L^2}^2+\norm{\na^k(\theta{\rm div}u)}_{L^2}^2\nonumber\\
&&\quad\le C_0\left((1+t)^{-(k+3+s)}+  (1+t)^{-(k+7/2+2s)} + (1+t)^{-(k+11/2+2s)} \right)\nonumber
\\&&\quad\le C_0(1+t)^{-(k+3+s)},
\end{eqnarray}where we required $N\ge2k+8+s$. Applying the Gronwall lemma to \eqref{npsitheta san} again, we obtain
\begin{equation*}
\mathcal{H}_k(t)\le \mathcal{H}_k(0)e^{-t}+C_0\int_0^te^{-(t-\tau)}(1+\tau)^{-(k+3+s)}\,d\tau\le  C_0(1+t)^{-(k+3+s)}.
\end{equation*}
This implies
\begin{equation}\label{npsitheta si}
\norm{\na^k(n,\theta,\psi)(t)}_{L^2}\lesssim \sqrt{\mathcal{H}_k(t) }\le C_0 (1+t)^{-\frac{k+3+s}{2}}.
\end{equation}
If required $N\ge2k+12+s$, then by \eqref{npsitheta si}, we have
\begin{equation*}
\norm{\na^{k+2}(n,\theta)(t)}_{L^2} \lesssim C_0(1+t)^{-\frac{k+5+s}{2}}.
\end{equation*}
Having obtained such faster decay, we can then improve  \eqref{npsi qi} to be
\begin{equation*}
 \frac{d}{dt}\mathcal{G}_k(t)+\mathcal{G}_k(t)
 \le C_0\left((1+t)^{-(k+5+s)}+  (1+t)^{-(k+7/2+2s)}  \right)
  \le C_0(1+t)^{-(k+7/2+2s)}.
\end{equation*}
Applying the Gronwall lemma again, we obtain
\begin{equation*}
\norm{\na^k(n,\psi)(t)}_{L^2}\lesssim \sqrt{\mathcal{G}_k(t) }\le C_0 (1+t)^{-(k/2+7/4+s)}.
\end{equation*}
In light of the faster decay for $\na^k\psi$, we can then improve  \eqref{npsitheta san} to be
\begin{equation*}
 \frac{d}{dt}\mathcal{H}_k(t)+\mathcal{H}_k(t)
 \le C_0\left((1+t)^{-(k+5+s)}+  (1+t)^{-(k+7/2+2s)}  \right)
  \le C_0(1+t)^{-(k+7/2+2s)}.
\end{equation*}
Applying the Gronwall lemma again, we obtain
\begin{equation*}
\norm{\na^k(n,\theta,\psi)(t)}_{L^2}\lesssim \sqrt{\mathcal{H}_k(t) }\le C_0 (1+t)^{-(k/2+7/4+s)}.
\end{equation*}
We thus complete the proof of \eqref{further decay2}. The proof of Theorem \ref{decay} is completed.


\begin{thebibliography}{10}

\bibitem{C} {\sc F. Chen},
{\em Introduction to Plasma Physics and Controlled Fusion. Vol. 1}, Plenum Press, New York, 1984.

\bibitem {CJW} {\sc G.~Q.~Chen, J.~W.~Jerome and D.~H.~Wang}, {\em Compressible
Euler-Maxwell equations}, Transp. Theory, Statist. Phys., 29 (2000), pp.
311--331.

\bibitem {D}
{\sc R. J. Duan}, {\em Global smooth flows for the compressible Euler-Maxwell system: The relaxation case}, J. Hyperbolic Differ. Equ., 8 (2011), pp. 375--413.

\bibitem {DLZ}
{\sc R. J. Duan, Q. Q. Liu and C. J. Zhu}, {\em The Cauchy problem on the compressible two-fluids Euler-Maxwell equations}, SIAM J. Math. Anal., 44(2012), pp. 102-133

\bibitem {FWK}
{\sc Y. H. Feng, S. Wang and S. Kawashima}, {\em  Global existence and asymptotic decay of solutions to the non-isentropic Euler-Maxwell system},
preprint (2012),  arXiv:1202.0111.

\bibitem {GM}
{\sc P. Germain, N. Masmoudi}, {\em  Global existence for the Euler-Maxwell system},
preprint (2011),  arXiv:1107.1595.


\bibitem{Gla} {\sc L. Grafakos},
{\em Classical and Modern Fourier Analysis}, Pearson
Education, Inc., Prentice Hall, 2004.

\bibitem{G}
{\sc Y. Guo}, {\em Smooth irrotational flows in the large to the Euler-Poisson system in $\mathbb{R}^{3+1}$},
Comm. Math. Phys.,  195 (1998), pp. 249--265.

\bibitem{G12}
{\sc Y. Guo},
{\em The Vlasov-Poisson-Landau system in a periodic box},
J. Amer. Math. Soc., 25 (2012), pp. 759--812.

\bibitem{GT}
{\sc Y. Guo, A. S. Tahvildar-Zadeh} {\em Formation of singularities in relativistic fluid dynamics and in spherically symmetric plasma dynamics}, In: Nonlinear partial differential equations (Evanston, IL, 1998), Contemp.Math., 238, Providence, RI: Amer. Math. Soc., 1999, pp. 151--161.

\bibitem{GW}
{\sc Y. Guo, Y. J. Wang},
{\em Decay of dissipative equations and negative Sobolev spaces},
Comm. Partial Differential Equations, 37 (2012), pp. 2165--2208.

\bibitem{HP}
{\sc M. L. Hajjej, Y. J. Peng}, {\em Initial layers and zero-relaxation limits of Euler-Maxwell equations}, J. Differential Equations 252 (2012), pp. 1441--1465.

\bibitem {J} {\sc J.~W.~Jerome}, {\em The Cauchy problem for compressible
hydrodynamic-Maxwell systems: A local theory for smooth solutions},
Differential Integral Equations, 16 (2003), pp. 1345--1368.

\bibitem{K} {\sc T. Kato}, {\em The Cauchy problem for quasi-linear
symmetric hyperbolic systems}, Arch. Ration. Mech. Anal., 58 (1975), pp.
181--205.

\bibitem{MB}{\sc A. J. Majda, A. L. Bertozzi}, {\em Vorticity and Incompressible Flow,
Cambridge University Press}, Cambridge, 2002.

\bibitem{MRS}
{\sc P. A. Markowich, C. Ringhofer and C. Schmeiser}, {\em
Semiconductor Equations}, Springer-Verlag: Vienna, 1990.

\bibitem{PW1}
 {\sc Y. J. Peng, S. Wang}, {\em Convergence of compressible
Euler-Maxwell equations to compressible Euler equations}, Comm.
Partial Differential Equations, 33 (2008), pp. 349--376.

\bibitem{PW2} {\sc Y. J. Peng, S. Wang}, {\em Rigorous derivation of
incompressible e-MHD equations from compressible Euler-Maxwell
equations}, SIAM J. Math. Anal., 40 (2008), pp. 540--565.

\bibitem{PWG}{\sc Y. J. Peng, S. Wang and Q. L. Gu },
{\em Relaxation limit and global existence of smooth solutions of compressible Euler-Maxwell equations},
SIAM J. Math. Anal., 43 (2011), pp. 944--970.

\bibitem{SS}
{\sc V. Sohinger, R. M. Strain},
{\em The Boltzmann equation, Besov spaces, and optimal time decay rates in $\mathbb{R}_{x}^{n}$}, preprint (2012), arXiv:1206.0027.

\bibitem{SG06}
{\sc R. M. Strain, Y. Guo},
{\em Almost exponential decay near Maxwellian},
Comm. Partial Differential Equations, 31 (2006), pp. 417--429.

\bibitem{TW}
{\sc Z. Tan, Y. J. Wang},
{\em Global existence and large-time behavior of weak solutions to the
compressible magnetohydrodynamic equations with Coulomb force},
Nonlinear Analysis:
Theory, Methods and Applications,  71 (11) (2009), pp. 5866--5884.

\bibitem{TWW}
{\sc Z. Tan, Y. J. Wang, Y. Wang},
{\em Global solution and time decay of the compressible Euler-Maxwell system in $\mathbb{R}^3$},
preprint (2012), arXiv:1207.2207

\bibitem{T1}
{\sc B. Texier}, {\em WKB asymptotics for the Euler-Maxwell equations}, Asymptot. Anal., 42 (2005), pp. 211--250.

\bibitem{T2}
{\sc B. Texier}, {\em Derivation of the Zakharov equations}, Arch. Ration. Mech. Anal., 184 (2007), pp. 121--183.

\bibitem{UK}
{\sc Y. Ueda, S. Kawashima}, {\em Decay property of regularity-loss type for the Euler-Maxwell system}, Methods Appl. Anal., 18 (2011), pp. 245--268.

\bibitem{UWK}
{\sc Y. Ueda, S. Wang and S. Kawashima}, {\em Dissipative structure of the regularity-loss type and time asymptotic decay of solutions for the Euler-Maxwell system}, SIAM J. Math. Anal., 44 (2012), pp. 2002--2017.

\bibitem{WFL}
{\sc S. Wang, Y. H. Feng and X. Li}, {\em The asymptotic behavior of globally smooth solutions of bipolar non-isentropic compressible Euler-Maxwell system for plasma}, preprint(2012), arXiv:1202.0112.

\bibitem{W12}
{\sc Y. J. Wang},
{\em  Global solution and time decay of the Vlasov-Poisson-Landau system in $\r3$}, SIAM J. Math. Anal., 44(2012), pp. 3281--3323.

\bibitem{W122}
{\sc Y. J. Wang},
{\em  Decay of the Navier-Stokes-Poisson equations},
J. Differential Equations, 253 (2012), pp. 273--297.

\bibitem {X} {\sc J. Xu},
{\em Global classical solutions to the compressible Euler-Maxwell equations}, SIAM J. Math. Anal., 43 (2011), pp. 2688--2718.

\bibitem {XXK} {\sc J. Xu, J. Xiong and S. Kawashima},
{\em Global well-posedness in critical Besov spaces for two-fluid Euler-Maxwell equations}, preprint(2012), arXiv:1208.3532.





\end{thebibliography}
\end{document}